\documentclass[preprint,10pt]{elsarticle}
\usepackage{amsfonts,amssymb,amsmath}
\usepackage{graphicx,color}
\usepackage{latexsym}
\usepackage{amsthm}
\usepackage{graphics}
\usepackage{subfigure}
\usepackage{epstopdf}
\usepackage{natbib}
\usepackage{float}
\usepackage{bm}
\usepackage{indentfirst}
\usepackage{algorithm}
\usepackage{algorithmic}
\usepackage{xcolor}
\allowdisplaybreaks[4]

\journal{Computers \& Mathematics with Application}

\newtheorem{theorem}{Theorem}

\begin{document}
\begin{frontmatter}

\title{A well-balanced positivity-preserving discontinuous Galerkin method for shallow water models with variable density}
\author[a]{Jun She}
\ead{shejun1014@163.com}
\author[a]{Haiyun Dong}
\ead{dhy@cqu.edu.cn}
\author[b]{Maojun Li}
\ead{limj@uestc.edu.cn}
\author[c]{Jianjun Ma\corref{cor1}}
\ead{majianjun425@163.com}

\address[a]{College of Mathematics and Statistics, Chongqing University, Chongqing, 401331, P.R. China }
\address[b]{School of arts and sciences, Chengdu College of University of Electronic Science and Technology of China, Sichuan, 611731, P.R. China}
\address[c]{College of Finance and Economics, Sichuan International Studies University, Chongqing 400031, P.R. China}

\cortext[cor1]{Corresponding author.}

\begin{abstract}
In this paper, we present a numerical scheme designed for coupled systems of variable-topography shallow water flow and solute transport. By integrating a variable-density system with an expression for relative density of mixtures, a novel formulation of the coupled system is derived. To ensure the well-balanced property, auxiliary variables are introduced to reformulate the variable-density shallow water equations into a new form, which is then discretized using the discontinuous Galerkin (DG) method with the Lax-Friedrichs (LF) flux as the numerical flux. By selecting appropriate values for the auxiliary variables, we demonstrate that the proposed method accurately preserves steady-state solutions under still water conditions, thereby verifying its well-balanced nature. Furthermore, sufficient conditions for preserving the positivity of both water depth and concentration are proposed and rigorously proven. A positivity-preserving limiter is introduced to enforce these conditions. Finally, a series of numerical examples are conducted to validate the computational accuracy and effectiveness of the proposed method.

\end{abstract}


\end{frontmatter}

\section{Introduction}\label{sec:intro}
Coupled systems of shallow water equations and solute transport models are commonly used to simulate the dynamic characteristics of water flow and pollutant transport in river and coastal zone environments\cite{Luis2012,Faranak2018}. Numerical simulation of this coupled model serves as a critical approach for assessing surface water pollution risks and finds wide application in environmental management and conservation practices\cite{Hu2020,Pablo2018}. Under overland flow conditions, free-surface shallow water flows often transport debris and suspended sediments. When such runoff discharges into larger water bodies like rivers or lakes, the notable density difference between the effluent and receiving water bodies must be considered in numerical modeling to improve the predictive accuracy of flow dynamics and mixing processes.

The development of numerical solutions for coupled models of shallow water equations and variable-density solute transport presents a significant challenge. This task demands robust numerical algorithms that ensure the scheme is well-balanced and maintains the positivity of certain key physical parameters. The well-balanced property calls for accurately preserving steady-state solutions and avoiding spurious oscillations caused by imbalanced discretization of source and flux terms. Positivity preservation requires that water depth, variable density, and scalar concentration remain non-negative throughout the entire computation. Although numerous studies have focused on source and flux balancing techniques at the discrete level\cite{Vázquez-Cendón1999,Thierry2003,Abdelaziz2016}, effectively applying these numerical methods under complex conditions while maintaining both numerical stability and physical consistency remains a considerable challenge.

Numerous studies in the literature have addressed numerical schemes for solving the variable-density shallow water equations. In 2017, Sepideh Khorshid et al.\cite{Khorshid2017} extended the central-upwind scheme for the variable-density shallow water equations to triangular meshes. The scheme demonstrates good balance and positivity, achieves high computational efficiency while maintaining accuracy, preserves two types of still-water steady states, and remains oscillation-free in regions with minor density variations. In 2020, Ernesto Guerrero Fernández et al.\cite{Guerrero Fernández2020}developed a second-order balanced finite volume scheme for variable-density multi-layer shallow water models, which integrates hydrostatic reconstruction and MUSCL second-order reconstruction to maintain balance for still-water steady-state solutions and preserve water depth positivity. In 2021, Amine Hanini et al.\cite{Hanini2021} proposed a well-balanced and positivity-preserving numerical scheme for variable-density shallow water models by enhancing the central-upwind scheme on triangular meshes to solve the coupled shallow water flow and scalar transport system, and demonstrated its balance and positivity properties. In 2022, Mohammad A. Ghazizadeh et al.\cite{Ghazizadeh2022} proposed an adaptive central-upwind scheme on quadtree grids for variable-density shallow water equations. This method accurately preserves the still-water steady state, utilizes continuous piecewise bilinear interpolation for bottom topography to achieve high-order spatial accuracy, and ensures positive water depth values.

The DG method is a powerful discretization technique for solving partial differential equations. It was first proposed in 1973 by Reed and Hill \cite{Reed1973} for solving neutron transport equations. The primary development of the DG method was accomplished by Cockburn and colleagues\cite{Cockburn1990,Cockburn1989-1,Cockburn1989-2,Cockburn1991,Cockburn1998-1}, then it has also evolved to include variants such as the local discontinuous Galerkin (LDG) method and the central discontinuous Galerkin (CDG) method. Numerous scholars have applied these methods to solve nonlinear shallow water wave equations and have designed a series of structure-preserving numerical schemes. Many scholars have applied this method to solve nonlinear shallow water wave equations and have designed a series of structure-preserving numerical schemes. For details, see \cite{Li2014,Li2017,Guerrero Fernández2022,Li2022,Zhang2023,Ersing2024,Bunya2009,Xing2010,Xian2021}.

In this paper, we propose a well-balanced and positivity-preserving DG method for simulating coupled systems of shallow water equations and solute transport models. The method is based on the fully coupled system introduced in \cite{Javier2012}. Referring to the methodological approach in \cite{Li2022}, we introduce an auxiliary parameter, and after performing an equivalent reconstruction of the coupled system, we employ the DG method based on the Lax-Friedrichs flux for numerical discretization. By selecting appropriate values for the auxiliary variable, the balance property of the scheme is proven by extending the reasoning approach of \cite{Li2022}. Additionally, sufficient conditions for positivity preservation are provided and verified. A positivity-preserving limiter is then introduced to ensure that both water depth and concentration satisfy these sufficient conditions.

The remainder of this paper is organized as follows: Section \ref{sec:varia} provides a brief introduction to the coupled system of shallow water equations and scalar transport model. Section \ref{sec:fully} presents the fully discrete DG numerical scheme for solving this system. The well-balanced properties of the proposed scheme are discussed in detail in Section \ref{sec:wellb}, followed by an analysis and proof of its positivity-preserving characteristics for water depth, mixture density, and scalar concentration in Section \ref{sec:posit}. The efficacy of the scheme is then validated through several numerical examples in Section \ref{sec:numer}. Section \ref{sec:concl} presents the conclusions of the study.

\section{Variable density shallow water system}\label{sec:varia}

In this paper, we consider the following coupled system of the shallow water equations and a solute transport model \cite{Xian2021}:
\begin{equation} \label{eq:1}
\left\{
\begin{aligned}
&\partial_t(hr)+\partial_x(hur)+\partial_y(hvr)=0, \\
&\partial_t(hur)+\partial_x\left(hu^2r+\frac{1}{2}gh^2r\right)+\partial_y(huvr)=-ghr\partial_xZ, \\
&\partial_t(hvr)+\partial_x(huvr)+\partial_y\left(hv^2r+\frac{1}{2}gh^2r\right)=-ghr\partial_yZ, \\
&\partial_t(hc_i)+\partial_x(huc_i)+\partial_y(hvc_i)=0,
\end{aligned}
\right.
\end{equation}
where \( h \) is the depth of the mixture, \( u \) and \( v \) are the flow velocities of the mixture in the \( x \)- and \( y \)- directions, respectively, \( Z \) represents the bottom topography, \( c_i \) is the scalar volume concentration of the \( i \)-th mixture (\( i=1,\dots,N \)), and \( g \) denotes the gravitational acceleration, \( t \) is the time.

The relative density of the mixture with respect to clear water is defined as follows:
\begin{equation} \label{eq:2}
	r=1+\sum_{i=1}^N\Delta_ic_i,
\end{equation}
where \( \Delta_i=(\rho_i-\rho_w)/\rho_w \) is the relative density of the \( i \)-th mixture, \( \rho_i \) denotes its density, \( \rho_w \) is the density of clear water, and \( \rho_wr \) represents the density of the mixture.

To facilitate the subsequent construction of the numerical scheme, we multiply both sides of  (\ref{eq:2}) by the water depth \( h \) and take the partial derivative with respect to time \( t \), obtaining:

$$\begin{aligned}\partial_t(hr)=\partial_t(h)+\sum_{i=1}^{N}\Delta_i\partial_t(hc_i).\end{aligned}$$

Then, by substituting the relevant equations from (\ref{eq:1}) into the above equation and simplifying, we obtain:

$$\begin{aligned} \partial_{t}(h) & +\partial_x(hu)+\partial_y(hv)=0. \end{aligned}$$

Since the bottom topography \( Z \) is time-independent, we can write it in the following form:
$$\begin{aligned} \partial_{t}(h+Z) & +\partial_x(hu)+\partial_y(hv)=0. \end{aligned}$$

Then, we reformulate the coupled system as:

\begin{equation} \label{eq:3}
\left\{
\begin{aligned}
&\partial_t(h+Z)+\partial_x(hu)+\partial_y(hv)=0, \\
&\partial_t(hr)+\partial_x(hur)+\partial_y(hvr)=0, \\
&\partial_t(hur)+\partial_x\left(hu^2r+\frac{1}{2}gh^2r\right)+\partial_y(hvr)=-ghr\partial_xZ, \\
&\partial_t(hvr)+\partial_x(huvr)+\partial_y\left(hv^2r+\frac{1}{2}gh^2r\right)=-ghr\partial_yZ, \\
&\partial_t(hc_i)+\partial_x(huc_i)+\partial_y(hvc_i)=0.
\end{aligned}
\right.
\end{equation}

To describe the reformulated coupled system better, we define the free surface elevation as \( \eta = h + z \), and rewrite the above system of equations in vector form:

\begin{equation}\label{eq:4}\partial_tU+\partial_xF(U)+\partial_yG(U)=S(U),\end{equation}
where $U=(\eta,p_1,p_2,p_3,q_1,q_2,...,q_N )^T$ , \( p_1 := hr \), \( p_2 := hur \),  \( p_3 := hvr \) and \( q_i := hc_i \)(\( i=1,\dots,N \)). Then, the flux term can be expressed as:
\begin{equation}\begin{gathered}\label{eq:5}
		F(U)=\left(\frac{(\eta-Z)p_2}{p_1},p_2,\frac{p_2^2}{p_1}+\frac{1}{2}g(\eta-Z)p_1,\frac{p_2p_3}{p_1},\frac{q_1p_2}{p_1},\frac{q_2p_2}{p_1},...,\frac{q_Np_2}{p_1}\right)^T, \\
		G(U)=\left(\frac{(\eta-Z)p_3}{p_1},p_3,\frac{p_2p_3}{p_1},\frac{p_3^2}{p_1}+\frac{1}{2}g(\eta-Z)p_1,\frac{q_1p_3}{p_1},\frac{q_2p_3}{p_1},...,\frac{q_Np_3}{p_1}\right)^T.
\end{gathered}\end{equation}

The source term is expressed as follows:
\begin{equation}\label{eq:6}S(U)=\left(0,0,-gp_1\partial_xZ,-gp_1\partial_yZ,0,0,...,0\right)^T.\end{equation}

\section{Fully discrete DG scheme}\label{sec:fully}

In this paper, let $\left\{I_i=(x_{i-1/2},x_{i+1/2})\right\}_{i=1}^{M_x}$ and  $\left\{J_j=(y_{j-1/2},y_{j+1/2})\right\}_{j=1}^{M_y}$ denote the uniform partitions of \((x_{min},x_{max})\) and \((y_{min},y_{max})\), respectively. and let \(\Delta x=x_{i+1/2}-x_{i-1/2}\) and \(\Delta y=y_{j+1/2}-y_{j-1/2}\) represent the lengths of \( I_i \) and \( J_j \), respectively.
\(\left\{C_{ij}=I_i\times J_j,\forall i,j \right\}\)is the uniform partition of space region \( (x_{min},x_{max}) \times (y_{min},y_{max}) \). Then we define a discrete finite dimensional function space to approximate \( U \) as follows:
\begin{equation}\label{eq:7}
	W_h=W_h^k=\left\{v:v\mid_{C_{ij}}\in\left[P^k(C_{ij})\right]^{N+4},\forall i,j\right\},
\end{equation}
where\( \left[ P^k(C_{ij}) \right]^{N+4} = \left\{V=(v_1,v_2,...,v_{N+4})^T:v_l\in P^k(C_{ij}),l=1,...,N+4\right\} \). \( P^k(C_{ij}) \) denotes a polynomial space defined on rectangle \( C_{ij} \) with degree of polynomial at most k.

Let \(\{t_n\}_{n=0}^{N_t}\) be a partition of the time domain \( (0, T) \), and suppose that we know the numerical solution of \( U \) at time \( t = t_n \) denoted by
\begin{equation}\label{eq:8}U_h^n=\left(\eta^n,(p_1)^n,(p_2)^n,(p_3)^n,(q_1)^n,(q_2)^n,...,(q_N)^n\right)^T\in W_h,\end{equation}
what we need to do is to find the numerical solution
\begin{equation}\label{eq:9}U_h^{n+1}=\left(\eta^{n+1},(p_1)^{n+1},(p_2)^{n+1},(p_3)^{n+1},(q_1)^{n+1},(q_2)^{n+1},...,(q_N)^{n+1}\right)^T\in W_h,\end{equation}
at \( t = t_{n+1} = t_n + \Delta t_n \).

We use the standard DG method for space discretization and the forward Euler method for time discretization to solve (\ref{eq:4}).That is to find \(U_h^{n+1}\in W_{h}^k\), such that \(\forall V\in W_{h}^k\),
\begin{equation}\begin{aligned}\label{eq:10}
		\int_{C_{ij}}U_{h}^{n+1}\cdot Vdxdy & =\int_{C_{ij}}U_h^n\cdot Vdxdy \\
		& +\Delta t_n\int_{C_{ij}}F(U_h^n)\cdot V_x+G(U_h^n)\cdot V_ydxdy \\
		& -\Delta t_n\int_{J_j}\hat{F}_{i+1/2}\cdot V(x_{i+1/2}^-,y)-\hat{F}_{i-1/2}\cdot V(x_{i-1/2}^+,y)dy \\
		& -\Delta t_n\int_{I_i}\hat{G}_{j+1/2}\cdot V(x,y_{j+1/2}^-)-\hat{G}_{j-1/2}\cdot V(x,y_{j-1/2}^+)dx \\
		& +\Delta t_n\int_{C_{ij}}S(U_h^n)\cdot Vdxdy.
\end{aligned}\end{equation}

In this scheme, \(\hat{F}_{i\pm1/2}\) and \(\hat{G}_{j\pm1/2}\) are the so-called numerical fluxes which is given by the Lax-Friedrichs flux:
\begin{equation}\begin{gathered}\label{eq:11}
		\hat{F}=\hat{F}^{LF}(U^-,U^+)=\frac{1}{2}(F(U^+)+F(U^-)-\alpha_1(U^+-U^-)),\\
		\hat{G}=\hat{G}^{LF}(U^-,U^+)=\frac{1}{2}(G(U^+)+G(U^-)-\alpha_2(U^+-U^-)),
\end{gathered}\end{equation}
with \(\alpha_1=\max(|u|+\sqrt{gh})\) and \(\alpha_2=\max(|v|+\sqrt{gh})\). The maximum is taken in the neighbor cells.

\section{Well-balanced DG reconstruction scheme}\label{sec:wellb}

In this section, we propose a well-balanced DG method and demonstrate its well-balanced properties. Additionally, we derive non-trivial steady-state solutions for system (\ref{eq:3}) and employ them for validation in the numerical examples of Section \ref{sec:numer}.

\subsection{Nontrivial steady state solutions}

In this section, we present a class of nontrivial steady state solutions for the system (\ref{eq:3}). We start by considering the steady state solution of stationary water, that is:
\begin{equation}\label{eq:12} h+Z=constant,u=v=0.\end{equation}

To extend the concept of equilibrium properties using the aforementioned solutions, we substitut (\ref{eq:12}) into (\ref{eq:3}) and obtain:

\begin{equation}\label{eq:13}\begin{gathered}
		\partial_{t}(hr)=\partial_{t}(hc_{i})=0,\\ \partial_{x}(\frac{1}{2}gh^{2}r)=-ghr\partial_{x}Z,\\
		\partial_{y}(\frac{1}{2}gh^{2}r)=-ghr\partial_{y}Z.
\end{gathered}\end{equation}

The relative density of the mixture can be shown to satisfy the following condition:\(\partial_tr=\partial_xr=\partial_yr=0\).

Thus, a hydrostatic solution of the system can be derived as follows:
\begin{equation}\label{eq:14}\eta=h+Z=C,u=v=0,r=r_b,c_i=c_b,\forall i=1,2,...,N,\end{equation}
where \( C \) , \( r_b \) and \( c_b \) are constants.

In (\ref{eq:14}), all the scalar concentrations of the mixtures are constants. In fact, we notice that the partial derivatives of the mixture concentration with respect to \( x \) and \( y \) are not zero, so we can use nontrivial steady-state solutions as follows:
\begin{equation}\label{eq:15}\eta=h+Z=C,u=v=0,\end{equation}
and the concentration functions of the components of the mixture depend only on space and satisfy the following condition.

\begin{equation}\label{eq:16}\sum_{i=1}^N\Delta_ic_i(x,y)=constant,\end{equation}
where we omitted the time variable$(c_i(x,y):=c_i(x,y,t))$.

This keeps the relative density of the mixture constant, and for mixtures with \( N\ge2 \)    components, the resulting solution is spatially steady-state but with variable concentrations.

\subsection{Well-balanced DG scheme}

In this subsection, we present a well-balanced DG method for the system (\ref{eq:4}). This method is similar to the well-balanced DG scheme in \cite{Li2022}, by introducing an auxiliary parameter and then rewriting the equations equivalently, this method does not alter the definition of the numerical flux. For simplicity, in this section and the following chapters, all vectors in (\ref{eq:4}) are represented by only the first five terms, where the fifth term of \( U \) is denoted by \( q_i \), and this does not affect our subsequent proof.. Then the system (\ref{eq:4}) can be rewritten as the following form
\begin{equation}\label{eq:17}\partial_tU+\partial_x\tilde{F}(U,B)+\partial_y\tilde{G}(U,B)=\tilde{S}(U,B),\end{equation}
where
\begin{equation}\label{eq:18}
\begin{aligned}
&	\tilde{F}(U,B)=\left(\frac{(\eta-Z)p_2}{p_1},p_2,\frac{p_2^2}{p_1}+(\frac{1}{2}g\eta^2-g(\eta-B)Z)\frac{p_1}{\eta-Z},\frac{p_2p_3}{p_1},\frac{q_ip_2}{p_1}\right)^T,\\
&	\tilde{G}(U,B)=\left(\frac{(\eta-Z)p_3}{p_1},p_3,\frac{p_2p_3}{p_1},\frac{p_3^2}{p_1}+(\frac{1}{2}g\eta^2-g(\eta-B)Z)\frac{p_1}{\eta-Z},\frac{q_ip_3}{p_1}\right)^T,\\
&	\tilde{S}(U,B)=\left(0,0,g(B-\eta)Z_x\frac{p_1}{\eta-Z}+g(\frac{p_1}{\eta-Z})_xZ(B-\frac{Z}{2}), g(B-\eta)Z_y\frac{p_1}{\eta-Z}
\right.\\
&\left.
+g(\frac{p_1}{\eta-Z})_yZ(B-\frac{Z}{2}),0\right)^T.\\
\end{aligned}
\end{equation}

Here, \( B \) is an auxiliary variable that only depends on time \( t \). We observe that after introducing this auxiliary variable, for the "still-water" solution (\ref{eq:15}) in the reformulated system (\ref{eq:17}), both the source and flux terms for \( hur \) and \( hvr \) in the equations equal to zero provided that \( B = C \). We continue to discretize system (\ref{eq:17}) using the standard DG method. That is to find \(U_h^{n+1}\in W_{h}^k\), such that \(\forall V\in W_{h}^k\),

\begin{equation}\label{eq:19}\begin{aligned}
		\int_{C_{ij}}U_{h}^{n+1}\cdot Vdxdy & =\int_{C_{ij}}U_{h}^{n}\cdot Vdxdy \\
		& +\Delta t_n\int_{C_{ij}}\tilde{F}(U_h^n,B_h^n)\cdot V_x+\tilde{G}(U_h^n,B_h^n)\cdot V_ydxdy \\
		& -\Delta t_n\int_{J_j}\hat{\tilde{F}}_{i+1/2}\cdot V(x_{i+1/2}^-,y)-\hat{\tilde{F}}_{i-1/2}\cdot V(x_{i-1/2}^+,y)dy \\
		& -\Delta t_n\int_{I_i}\hat{\tilde{G}}_{j+1/2}\cdot V(x,y_{j+1/2}^-)-\hat{\tilde{G}}_{j-1/2}\cdot V(x,y_{j-1/2}^+)dx \\
		& +\Delta t_n\int_{C_{ij}}\tilde{S}(U_h^n,B_h^n)\cdot Vdxdy.
\end{aligned}\end{equation}

As in (\ref{eq:10}), the numerical fluxes \( \hat{\tilde{F}} \) and \( \hat{\tilde{G}} \) are also given by the Lax-Friedrichs flux(\ref{eq:11}):

\begin{equation}\label{eq:20}\hat{\tilde{F}}=\hat{\tilde{F}}^{LF}(U^-,U^+;B_h^n),\hat{\tilde{G}}=\hat{\tilde{G}}^{LF}(U^-,U^+;B_h^n).\end{equation}

In the scheme (\ref{eq:19}), \(B_h^n\) is a numerical approximation at time \(t_n\) of the auxiliary variable \(B\) in (\ref{eq:19}). If we design a well-balanced DG method, then for the "still-water" solution, \( B_h^n \) should be a constant, and this constant equals \( C \) in (\ref{eq:15}). Therefore, we can define \( B_h^n \) as follows:

\begin{equation}\label{eq:21}B_h^n=\frac{1}{|\Omega|}\int_\Omega\eta^ndxdy.\end{equation}

For "still-water" solutions, the initial variables \( hur \) and \( hvr \) are not constant but depend on space. If we use the \( L^2 \) projection of \( U \), the resulting approximate solution \( U_h \) becomes discontinuous at element boundaries, which would disrupt the well-balanced property of scheme (\ref{eq:19}). To address this issue, we approximate the initial variable \( U \) using piecewise \( k \)-th degree polynomial interpolation. This ensures that the approximate function \( U_h \) is continuous at element boundaries. Additionally, while \( k = 1 \) or \( k = 2 \), the interpolation nodes include all boundary points, so the values of \( U_h \) on the element boundaries match those of the initial function, thereby preserving the well-balanced property of scheme (\ref{eq:19}).

Now, we have the following theorem.

\begin{theorem}
	The proposed DG method defined in $(\ref{eq:19})$-$(\ref{eq:21})$ is a well-balanced scheme, meaning it preserves the “still-water” solution $(\ref{eq:15})$-$(\ref{eq:16})$. 
\end{theorem}
\textbf{Proof.}
We consider the reformulated form (\ref{eq:17})-(\ref{eq:18}) and the new system (\ref{eq:19}) satisfied by \( U \),now we use the mathematical induction to prove \(U^{n+1}=U^n,\forall n\geq1\).

We assume that the water wave state at the initial time is a “still-water” solution in (\ref{eq:15})-(\ref{eq:16}), then we can easily get
\begin{equation}\label{eq:22}U_h^0=\left(C,(p_1)^0,0,0,(q_i)^0\right)^T,r^0=r_b,\end{equation}
where \( r_b \) is the constant defined in (\ref{eq:14}).

Suppose that (\ref{eq:22}) holds for \( n = k \), namely
\begin{equation}\label{eq:23} U_h^k=\left(C,(p_1)^0,0,0,(q_i)^0\right)^T,r^0=r_b,\end{equation}
thus we have

\begin{equation}\begin{gathered}\label{eq:24}
		U_h^k(x_{i\pm1/2}^\pm,y)=\left(C,(p_1)_{i\pm1/2,j}^{0,\pm},0,0,(q_i)_{i\pm1/2,j}^{0,\pm}\right)^T,\\
		U_h^k(x,y_{j\pm1/2}^\pm)=\left(C,(p_1)_{i,j\pm1/2}^{0,\pm},0,0,(q_i)_{i,j\pm1/2}^{0,\pm}\right)^T.
\end{gathered}\end{equation}

Depending on the choice of initial values, we conclude that \((p_1)^{0,+}=(p_1)^{0,-}\) and \((q_i)^{0,+}=(q_i)^{0,-}\).

Then, we can obtain \(U_h^k(x_{i\pm1/2}^+,y)=U_h^k(x_{i\pm1/2}^-,y)\) and \(U_h^k(x,y_{j\pm1/2}^+)=U_h^k(x,y_{j\pm1/2}^-)\).

The definition of \(B_h^k\) implies \(B_h^k=C\), thus we have from (\ref{eq:18}):
\begin{equation}\begin{aligned}\label{eq:25}
		&	\tilde{F}(U_{h}^{k},B_{h}^{k})=\left(0,0,\frac{1}{2}gC^{2}r_{b},0,0\right)^{T},\\
		&	\tilde{G}(U_{h}^{k},B_{h}^{k})=\left(0,0,0,\frac{1}{2}gC^{2}r_{b},0\right)^{T},\\
		&	\tilde{S}(U_{h}^{k},B_{h}^{k})=\left(0,0,0,0,0\right)^{T}.\\
\end{aligned}\end{equation}

Combining Eqs. (\ref{eq:11}) and (\ref{eq:24})-(\ref{eq:25}), the numerical fluxes are

\begin{equation}\begin{gathered}\label{eq:26}
		\hat{\tilde{F}}_{i\pm1/2}=\hat{\tilde{F}}^{LF}(U_h^k(x_{i\pm1/2}^-,y),U_h^k(x_{i\pm1/2}^+,y);B_h^k)=\left(0,\frac{1}{2}gC^2r_b,0\right)^T,\\
		\hat{\tilde{G}}_{j\pm1/2}=\hat{\tilde{G}}^{LF}(U_h^k(x,y_{j\pm1/2}^-),U_h^k(x,y_{j\pm1/2}^+);B_h^k)=\left(0,0,\frac{1}{2}gC^2r_b\right)^T.
\end{gathered}\end{equation}

Plugging Eqs. (\ref{eq:25})-(\ref{eq:26}) and into (\ref{eq:19}) with \( n = k \), we can get
\begin{equation}\label{eq:27}
	\int_{C_{ij}}U_h^{k+1}\cdot Vdxdy=\int_{C_{ij}}U_h^k\cdot Vdxdy,(i=1,2,...,M_x,j=1,2,...,M_y).
\end{equation}
Therefore,
\begin{equation}\label{eq:28}
	\int_{C_{ij}}(U_h^{k+1}-U_h^k)\cdot Vdxdy=0,(i=1,2,...,M_x,j=1,2,...,M_y).
\end{equation}
Set \(V=U_h^{k+1}-U_h^k\),we obtain \(U_h^{k+1}=U_h^k\).This completes the proof.$\hfill\square$

\section{Positivity-Preserving DG scheme}\label{sec:posit}

In this section, we propose a positivity-preserving DG method for shallow water models with variable density , with which the cell averages of the computed water depth and concentrations are non-negative at any discrete time \( t = t_n \) .

\subsection{The cell averages of variables \( h \) and \( q_i \)}
To obtain the element-averaged numerical solution of variables in the DG method, by setting the test function \(V=(1,1,1,1,1)^{T}\) in scheme (\ref{eq:19}), the average value \(U_{h}^{n+1}\) of \(\bar{U}^{n+1}_{h,ij}\) in each element at time \( t_{n+1} \) satisfies the following equation:

\begin{equation}\begin{aligned}\label{eq:29}
		\bar{U}_{h,ij}^{n+1} & =\bar{U}_{h,ij}^{n} \\
		& -\frac{\Delta t_{n}}{\Delta x\Delta y}\int_{J_{j}}(\hat{\tilde{F}}_{i+1/2}-\hat{\tilde{F}}_{i-1/2})dy-\frac{\Delta t_{n}}{\Delta x_{i}\Delta y}\int_{I_{i}}(\hat{\tilde{G}}_{j+1/2}-\hat{\tilde{G}}_{j-1/2})dx \\
		& +\frac{\Delta t_{n}}{\Delta x\Delta y}\int_{C_{ij}}\tilde{S}(U_{h}^{n},B_{h}^{n})dxdy.
\end{aligned}\end{equation}

Therefore, the average value \( h \) of \(\bar{h}^{n+1}_{ij}\) at time \( t_{n+1} \) can be derived as

\begin{equation}\begin{aligned}\label{eq:30}
		\bar{h}_{ij}^{n+1} & =\bar{h}_{ij}^{n} \\
		& -\frac{\Delta t_n}{\Delta x\Delta y}\int_{J_j}\left[\hat{f}\left(H_1(x_{i+1/2}^-,y),H_1(x_{i+1/2}^+,y)\right)
		\right.\\
		&\left.
		-\hat{f}\left(H_1(x_{i-1/2}^-,y),H_1(x_{i-1/2}^+,y)\right)\right]dy \\
		& -\frac{\Delta t_{n}}{\Delta x\Delta y}\int_{I_{i}}\left[\hat{g}\left(H_2(x,y_{j+1/2}^{-}),H_2(x,y_{j+1/2}^{+})\right)
		\right.\\
		&\left.
		-\hat{g}\left(H_2(x,y_{j-1/2}^{-}),H_2(x,y_{j-1/2}^{+})\right)\right]dx,
\end{aligned}\end{equation}
where
\begin{equation}\begin{aligned}\label{eq:31}
		& \hat{f}\left(H_1(x_{i\pm1/2}^{-},y),H_1(x_{i\pm1/2}^{+},y)\right)\\
		&=\frac{1}{2}\left[hu(x_{i\pm1/2}^{-},y)+hu(x_{i\pm1/2}^{+},y)
		-\alpha_{1}\left(h(x_{i\pm1/2}^{+},y)-h(x_{i\pm1/2}^{-},y)\right)\right], \\
		& \hat{g}\left(H_2(x,y_{j\pm1/2}^{-}),H_2(x,y_{j\pm1/2}^{+})\right)\\
		&=\frac{1}{2}\left[hv(x,y_{j\pm1/2}^{-})+hv(x,y_{j\pm1/2}^{+})
		-\alpha_{2}\left(h(x,y_{j\pm1/2}^{+})-h(x,y_{j\pm1/2}^{-})\right)\right].
\end{aligned}\end{equation}

Similarly, the equation for \( q_i \) can be derived through an analogous procedure, the detailed formulation of which is omitted here for conciseness.

\subsection{Sufficient conditions for positivity-preserving DG schemes}

In this section, we present sufficient conditions for two-dimensional high-order positivity-preserving DG schemes. First, we consider the one-dimensional first-order DG formulation: let \(h_{i}^{n}(i=1,2,...,M_{x})\) denote the numerical solution of water depth \( h \) in the \( i \)-th element at time \( t_n \). From Equation (\ref{eq:30})-(\ref{eq:31}), we derive the first-order DG scheme as follows:
\begin{equation}\begin{aligned}\label{eq:32}
		h_{i}^{n+1} & =h_{i}^{n}-\frac{\Delta t_{n}}{\Delta x}\left[\hat{f}(h_{i}^{n},u_{i}^{n};h_{i+1}^{n},u_{i+1}^{n})-\hat{f}(h_{i-1}^{n},u_{i-1}^{n};h_{i}^{n},u_{i}^{n})\right] \\
		& =h_i^n-\frac{1}{2}\cdot\frac{\Delta t_n}{\Delta x}\left[\left(\left(hu\right)_i^n+\left(hu\right)_{i+1}^n-\alpha(h_{i+1}^n-h_i^n)\right)
		\right.\\
		&\left.
		-\left(\left(hu\right)_{i-1}^n+\left(hu\right)_i^n-\alpha(h_i^n-h_{i-1}^n)\right)\right].
\end{aligned}\end{equation}

Therefore,we have the following lemma.
\newtheorem{lemma}{Lemma}
\begin{lemma}\label{Lm.1}
	For the first-order scheme $(\ref{eq:32})$, if the following conditions hold at time \( t_n: \)
	\begin{enumerate}
		\itemsep=0pt
		\item \(h_{i}^{n}\geq0(i=1,2,...,M_{x});\)
		\item \(C_{CFL}=\lambda_{x}\alpha\leq1;\)
	\end{enumerate} 
	where \(\lambda_{x}=\Delta t_{n}/\Delta x,\alpha=\max_{x_{min}\leq x\leq x_{max}}(\mid u\mid+\sqrt{gh})\)
	
	Then \(h_{i}^{n+1}\geq0(i=1,2,...,M_{x})\) at \( t_{n+1} \).
\end{lemma}
\textbf{Proof.}
From the scheme (\ref{eq:32}), we can derive the following:
\begin{equation}\begin{aligned}\label{eq:33}
		h_{i}^{n+1} &
		=\left(1-\frac{\Delta t_n}{\Delta x}\cdot\alpha\right)h_i^n+\frac{1}{2}\cdot\frac{\Delta t_n}{\Delta x}\left[\left(\alpha-u_{i+1}^n\right)h_{i+1}^n+\left(\alpha+u_{i-1}^n\right)h_{i-1}^n.\right]
\end{aligned}\end{equation}

Then we can easily get \( \alpha-u_{i+1}^{n}\geq0,\alpha+u_{i-1}^{n}\geq0 \) and \(1-\frac{\Delta t_{n}}{\Delta x}\cdot\alpha\geq0\), additionally, note that \(h_{i}^{n}\geq0,h_{i-1}^{n}\geq0,h_{i+1}^{n}\geq0\), so we obtain \(h_{i}^{n+1}\geq0\) from (\ref{eq:33}).$\hfill\square$

Next, we consider the two-dimensional \( k \)-th order DG scheme. We begin by defining the following components.

Firstly, let \(S_i^x=\{x_i^\beta:\beta=1,\cdots,L\}\) and \(S_i^y=\{y_i^\beta:\beta=1,\cdots,L\}\) be the Gaussian quadrature points on \(I_{i}=\left[x_{i-1/2},x_{i+1/2}\right]\) and \(J_{j}=\left[y_{j-1/2},y_{j+1/2}\right]\) respectively. The corresponding quadrature weights \(\omega_{\beta}(\beta=1,2,...,L)\) are distributed on the interval \(\left[-1/2,1/2\right]\), \(\sum_{\beta=1}^{L}\omega_{\beta}=1\). Then let \(\hat{S}_{i}^{x}=\{\hat{x}_{i}^{r}:r=1,\cdots,M\}\) and \(\hat{S}_{j}^{y}=\{\hat{y}_{j}^{r}:r=1,\cdots,M\}\) be the Gauss-Lobatto quadrature quadrature points on \(I_{i}=\left[x_{i-1/2},x_{i+1/2}\right]\) and \(J_{j}=\left[y_{j-1/2},y_{j+1/2}\right]\) respectively. The corresponding quadrature weights \(\hat{\omega}_r(r=1,2,...,M)\) are distributed on the interval \(\left[-1/2,1/2\right]\), \(\sum_{r=1}^{N}\hat{\omega}_{r}=1\), where \( M \) satisfies \(2M-3>k\).

Then we have the following theorem.
\begin{theorem}\label{Th.2}
	Considering the equation $(\ref{eq:30})$ for the cell-averaged water depth \( \bar{h}, \) if the following conditions hold at time \( t_n: \)
	\begin{enumerate}
		\itemsep=0pt
		\item \(\bar{h}_{ij}^n\geq0;h_{ij}^n(x_i^\beta,\hat{y}_j^r)\geq0,h_{ij}^n(\hat{x}_i^r,y_j^\beta)\geq0,\forall i,j,r,\beta;\)
		\item \(C_{CFL}=\frac{\Delta t_{n}}{\Delta x}\alpha_{1}+\frac{\Delta t_{n}}{\Delta y}\alpha_{2}\leq\hat{\omega}_{1};\)
	\end{enumerate} 
	where \(\alpha_{1}=\|\mid u\mid+\sqrt{gh}\|_{\infty}\) and  \(\alpha_{2}=\|\mid v\mid+\sqrt{gh}\|_{\infty}\).
	
	Then, we have \(\bar{h}_{ij}^{n+1}\geq0(i=1,2,...,M_{x},j=1,2,...,M_{y})\) at \( t_{n+1} \).
\end{theorem}
\textbf{Proof.}
Let \(\lambda_{1}=\Delta t_{n}/\Delta x, \lambda_{2}=\Delta t_{n}/\Delta y\). Then we use the Gauss-Lobatto quadrature rule, to transform (\ref{eq:30}) into the following form:

\begin{equation}\begin{aligned}\label{eq:34}
		\bar{h}_{ij}^{n+1} & =\bar{h}_{ij}^{n} \\
		& -\lambda_{1}\Bigg[\sum_{\beta=1}^{L}\hat{f}\left(H_1(x_{i+1/2}^{-},y_{j}^{\beta}),H_1(x_{i+1/2}^{+},y_{j}^{\beta})\right)
		\\
		&
			-\hat{f}\left(H_1(x_{i-1/2}^{-},y_{j}^{\beta}),H_1(x_{i-1/2}^{+},y_{j}^{\beta})\right)\Bigg]\omega_{\beta} \\
		& -\lambda_2\Bigg[\sum_{\beta=1}^L\hat{g}\left(H_2(x_i^\beta,y_{j+1/2}^-),H_2(x_i^\beta,y_{j+1/2}^+)\right)
		\\
		&
		-\hat{g}\left(H_2(x_i^\beta,y_{j-1/2}^-),H_2(x_i^\beta,y_{j-1/2}^+)\right)\Bigg]\omega_\beta.
\end{aligned}\end{equation}

Let \(\mu=\alpha_{1}\lambda_{1}+\alpha_{2}\lambda_{2}\), we have
\begin{equation}\begin{aligned}\label{eq:35}
		\bar{h}_{ij}^{n} & =\frac{\alpha_{1}\lambda_{1}}{\mu}\bar{h}_{ij}^{n}+\frac{\alpha_{2}\lambda_{2}}{\mu}\bar{h}_{ij}^{n} \\
		& =\frac{\alpha_{1}\lambda_{1}}{\mu}\sum_{r=1}^{N}\sum_{\beta=1}^{L}\omega_{\beta}\hat{\omega}_{r}h_{ij}(\hat{x}_{i}^{r},y_{j}^{\beta})+\frac{\alpha_{2}\lambda_{2}}{\mu}\sum_{r=1}^{N}\sum_{\beta=1}^{L}\omega_{\beta}\hat{\omega}_{r}h_{ij}(x_{i}^{\beta},\hat{y}_{j}^{r}) \\
		& =\frac{\alpha_{1}\lambda_{1}}{\mu}\sum_{r=2}^{N-1}\sum_{\beta=1}^{L}\omega_{\beta}\hat{\omega}_{r}h_{ij}(\hat{x}_{i}^{r},y_{j}^{\beta})
		\\
		&
		+\frac{\alpha_{1}\lambda_{1}}{\mu}\hat{\omega}_{1}\sum_{\beta=1}^{L}\omega_{\beta}\left(h_{ij}(x_{i-1/2}^{+},y_{j}^{\beta})+h_{ij}(x_{i+1/2}^{-},y_{j}^{\beta})\right) \\
		& +\frac{\alpha_{2}\lambda_{2}}{\mu}\sum_{r=2}^{N-1}\sum_{\beta=1}^{L}\omega_{\beta}\hat{\omega}_{r}h_{ij}(x_{i}^{\beta},\hat{y}_{j}^{r})
		\\
		&
		+\frac{\alpha_{2}\lambda_{2}}{\mu}\hat{\omega}_{1}\sum_{\beta=1}^{L}\omega_{\beta}\left(h_{ij}(x_{i}^{\beta},y_{j-1/2}^{+})+h_{ij}(x_{i}^{\beta},y_{j+1/2}^{-})\right).
\end{aligned}\end{equation}

Here we have used \( \hat{\omega}_1 = \hat{\omega}_N \). Substituting (\ref{eq:35}) into (\ref{eq:34}) yields

\begin{equation}\begin{aligned}\label{eq:36}
		\bar{h}_{ij}^{n+1} & =\frac{\alpha_{1}\lambda_{1}}{\mu}\sum_{r=2}^{N-1}\sum_{\beta=1}^{L}\omega_{\beta}\hat{\omega}_{r}h_{ij}(\hat{x}_{i}^{r},y_{j}^{\beta})+\frac{\alpha_{2}\lambda_{2}}{\mu}\sum_{r=2}^{N-1}\sum_{\beta=1}^{L}\omega_{\beta}\hat{\omega}_{r}h_{ij}(x_{i}^{\beta},\hat{y}_{j}^{r}) \\
		& +\frac{\alpha_{1}\lambda_{1}}{\mu}\hat{\omega}_{1}\sum_{\beta=1}^{L}\omega_{\beta}\Bigg[h_{ij}(x_{i+1/2}^{-},y_{j}^{\beta})
		\\
		&
		-\frac{\mu}{\alpha_{1}\hat{\omega}_{1}}\bigg[\hat{f}\left(H_1(x_{i+1/2}^{-},y_{j}^{\beta}),H_1(x_{i+1/2}^{+},y_{j}^{\beta})\right)  \\
		& 
		-\hat{f}\left(H_1(x_{i-1/2}^+,y_j^\beta),H_1(x_{i+1/2}^-,y_j^{\beta})\right)\bigg]\Bigg] \\
		& +\frac{\alpha_{1}\lambda_{1}}{\mu}\hat{\omega}_{1}\sum_{\beta=1}^{L}\omega_{\beta}\Bigg[h_{ij}(x_{i-1/2}^{+},y_{j}^{\beta})
		\\
		&
		-\frac{\mu}{\alpha_{1}\hat{\omega}_{1}}\bigg[\hat{f}\left(H_1(x_{i-1/2}^{+},y_{j}^{\beta}),H_1(x_{i+1/2}^{-},y_{j}^{\beta})\right) \\
		&  -\hat{f}\left(H_1(x_{i-1/2}^{-},y_{j}^{\beta}),H_1(x_{i-1/2}^{+},y_{j}^{\beta})\right)\bigg]\Bigg] \\
		& +\frac{\alpha_{2}\lambda_{2}}{\mu}\hat{\omega}_{1}\sum_{\beta=1}^{L}\omega_{\beta}\Bigg[h_{ij}(x_{i}^{\beta},y_{j+1/2}^{-})\\
		&
		-\frac{\mu}{\alpha_{2}\hat{\omega}_{1}}\bigg[\hat{g}\left(H_2(x_{i}^{\beta},y_{j+1/2}^{-}),H_2(x_{i}^{\beta},y_{j+1/2}^{+})\right)\\
		& -\hat{g}\left(H_2(x_{i}^{\beta},y_{j-1/2}^{+}),H_2(x_{i}^{\beta},y_{j+1/2}^{-})\right)\bigg]\Bigg]\\
		&
		+\frac{\alpha_{2}\lambda_{2}}{\mu}\hat{\omega}_{1}\sum_{\beta=1}^{L}\omega_{\beta}\Bigg[h_{ij}(x_{i}^{\beta},y_{j-1/2}^{+})\\
		&
		-\frac{\mu}{\alpha_{2}\hat{\omega}_{1}}\bigg[\hat{g}\left(H_2(x_{i}^{\beta},y_{j-1/2}^{+}),H_2(x_{i}^{\beta},y_{j+1/2}^{-})\right) \\
		& -\hat{g}\left(H_2(x_{i}^{\beta},y_{j-1/2}^{-}),H_2(x_{i}^{\beta},y_{j-1/2}^{+})\right)\bigg]\Bigg].
\end{aligned}\end{equation}

By Lemma \ref{Lm.1}, we conclude that \(\bar{h}_{ij}^{n+1}\geq0\). This completes the proof.$\hfill\square$

Similar to the theorem above, analogous results for \((\bar{q_i})_{ij}^{n+1}\geq0\) can be established, with the detailed derivations omitted here.

\subsection{Positivity-Preserving limiter}
Next, under the assumption \(\bar{h}_{ij}^n\geq0\), \(\forall i,j\), we will give a positivity-preserving limiter which modifies the DG solution polynomials \(h_h\) at time tn into \(\tilde{h}_h\) such that they satisfy the sufficient condition in Theorem \ref{Th.2}, while maintaining accuracy and local conservation.

At time \( t_n \), the numerical solution \( h^n \) is modified on each grid cell \( C_{ij} \) using the following limiter \cite{Xing2010}:

\begin{equation}\label{eq:37}\tilde{h}_{ij}^n(x,y)=\delta_j\left(h_{ij}^n(x,y)-\bar{h}_{ij}^n\right)+\bar{h}_{ij}^n,\end{equation}
where 

\begin{equation}\begin{gathered}\label{eq:38}
		\delta_{j}=\min\left\{1,\frac{\bar{h}_{ij}^{n}}{\bar{h}_{ij}^{n}-m_{ij}}\right\},m_{ij}=\min_{(x,y)\in S_{ij}}h_{ij}^{n}(x,y),\\
		S_{ij}=\left\{(x,y):x\in S_{i}^{x},y\in\hat{S}_{j}^{y},or x\in\hat{S}_{i}^{x},y\in S_{j}^{y}\right\}.
\end{gathered}\end{equation}

Similarly, the same correction procedure is applied to the numerical solution \( (q_i)^n \).

\subsection{Positivity-Preserving Well-Balanced DG Methods}
By applying the positivity-preserving limiter (\ref{eq:37}) to modify \( U_h^n \) in (\ref{eq:19}), we obtain the revised solution denoted as \(\tilde{U}_h^n=\left(\tilde{\eta}^n,(p_1)^n,(p_2)^n,(p_3)^n,(\tilde{q_i})^n\right)^T\). 
Then the well-balanced DG scheme (\ref{eq:19}) transforms into: find \(U_h^{n+1}\in W_{h}^k\), such that \(\forall V\in W_{h}^k\),

\begin{equation}\begin{aligned}\label{eq:39}
		\int_{C_{ij}}U_{h}^{n+1}\cdot Vdxdy & =\int_{C_{ij}}\tilde{U}_{h}^{n}\cdot Vdxdy \\
		& +\Delta t_{n}\int_{C_{ij}}\tilde{F}(\tilde{U}_{h}^{n},B_{h}^{n})\cdot V_{x}+\tilde{G}(\tilde{U}_{h}^{n},B_{h}^{n})\cdot V_{y}dxdy \\
		& -\Delta t_{n}\int_{J_j}\hat{\tilde{F}}_{i+1/2}\cdot V(x_{i+1/2}^{-},y)-\hat{\tilde{F}}_{i-1/2}\cdot V(x_{i-1/2}^{+},y)dy \\
		& -\Delta t_{n}\int_{I_i}\hat{\tilde{G}}_{j+1/2}\cdot V(x,y_{j+1/2}^{-})-\hat{\tilde{G}}_{j-1/2}\cdot V(x,y_{j-1/2}^{+})dx \\
		& +\Delta t_{n}\int_{C_{ij}}\tilde{S}(\tilde{U}_{h}^{n},B_{h}^{n})\cdot Vdxdy.
\end{aligned}\end{equation}

Under the premise that all terms in the scheme maintain the same computational methodology as previously defined. The selection of an appropriate CFL number complying with the CFL condition in Theorem \ref{Th.2} enables the positivity-preserving limiter to guarantee non-negative values of the numerical solutions \( h^n \) and \( (q_i)^n \) at quadrature nodes. This establishes scheme (\ref{eq:39}) as a positivity-preserving and well-balanced DG method.

\section{Numerical tests}\label{sec:numer}

This section demonstrates the applicability of the proposed numerical scheme through test cases simulating coupled shallow water flow an
d solute transport. The finite-dimensional function space is constructed using a \( P^k \) approximation with \( k = 1 \) or \( 2 \), while time discretization employs a third-order TVD Runge-Kutta scheme instead of the forward Euler method. The method is formulated as a convex combination of the forward Euler method, thereby effectively preserving the balance property and non-negativity of the numerical scheme. In all numerical examples, we used the gravitational acceleration \( g = 1 \). The grid employs constant spatial steps \( \Delta x \) and \( \Delta y \) in the \( x \)- and \( y \)- directions respectively, with temporal step sizes determined as follows:

\begin{equation}\label{eq:40}\Delta t_n=C_{CFL}\frac{\min(\Delta x,\Delta y)}{\max\left(\left|u\right|+\sqrt{gh}\right)+\max\left(\left|v\right|+\sqrt{gh}\right)}.\end{equation}

To suppress numerical oscillations, a total variation diminishing (TVD) minmod slope limiter \cite{Cockburn1998-2} is employed in the computation.


\subsection{Small perturbation of a “still-water” state}

This section adopts the shallow water equations to validate the stability of the proposed numerical scheme and its ability to accurately capture the propagation of small perturbations from a steady-state solution. The computational domain in this test is \([0,2]\times[0,1]\). The test is conducted by introducing a small disturbance to a still-water equilibrium over a variable bottom topography with an exponential profile, where the bottom topography is given by

\begin{equation}\label{eq:41}Z(x,y)=0.8\exp\left(-5(x-0.9)^2-50(y-0.5)^2\right).\end{equation}

The initial condition for the water depth is as follows:

\begin{equation}\label{eq:42}\eta(0,x,y)=
	\begin{cases}
		1.01 & 0.05\leq x\leq0.15, \\
		1 & \mathrm{otherwise} .
\end{cases}\end{equation}

In this experiment, the vertically averaged horizontal velocity in the \( x- \) and \( y- \) directions are equal to zero. We set the wall boundary conditions in the \( x- \) direction while the outflow boundary conditions are imposed in the \( y- \) direction.

The DG method in (\ref{eq:39}) is used to solve this problem on a $200 \times 50$ mesh with \( CFL=0.3 \). Figure \ref{FIG:1} presents 2D and 3D contour plots of the water surface elevation at different times (\( t \) = 0.5, 0.9, 1.3 and 1.7). The results exhibit no numerical oscillations and clearly demonstrate the capability of the proposed scheme to accurately capture the dynamic evolution of small lake perturbations under “still-water” state, and
the results are very consistent with the results in \cite{Hanini2021}.
\begin{figure}
	\centering
	\includegraphics[width=.65\columnwidth]{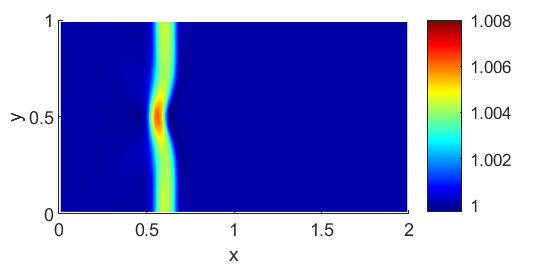}
	\includegraphics[width=.65\columnwidth]{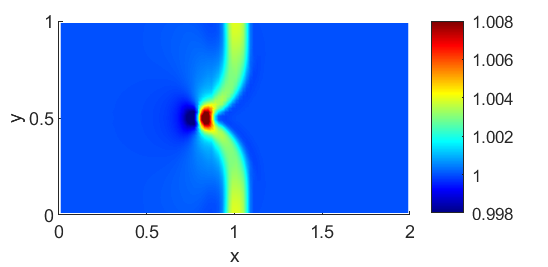}
	\includegraphics[width=.65\columnwidth]{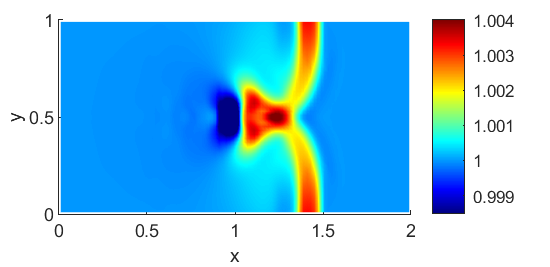}
	\includegraphics[width=.65\columnwidth]{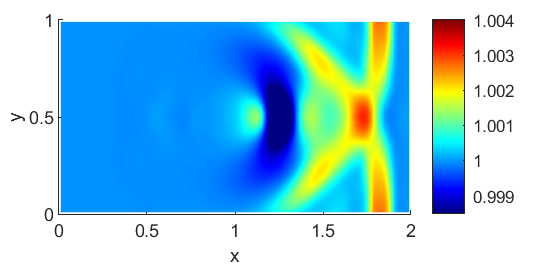}
	\caption{Two-dimensional views of the computed water surface elevation using the proposed schemes at times \( t = 0.5, 0.9, 1.3 \) and \(1.7 \).}
	\label{FIG:1}
\end{figure}

\begin{figure}
	\centering
	\includegraphics[width=0.35\columnwidth]{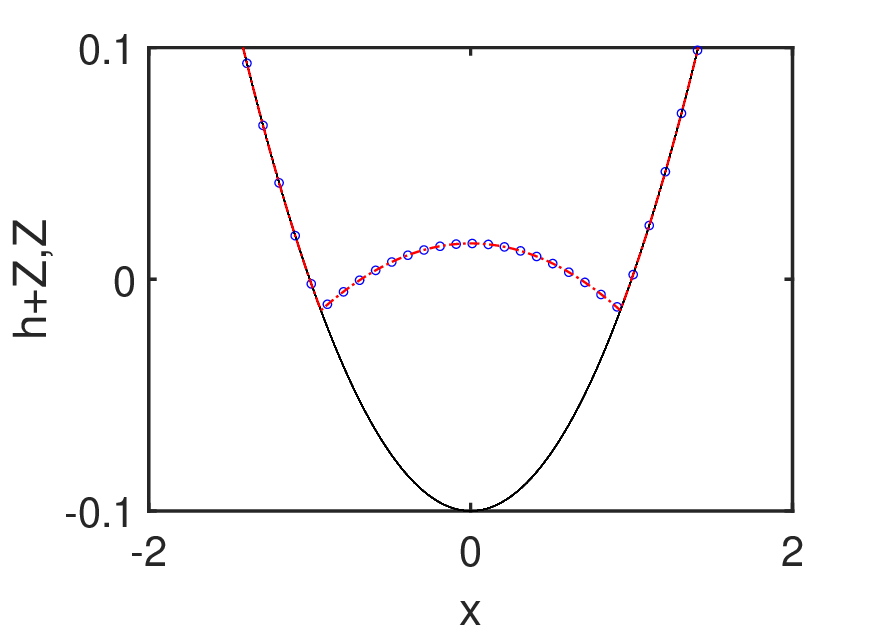}
	\includegraphics[width=0.35\columnwidth]{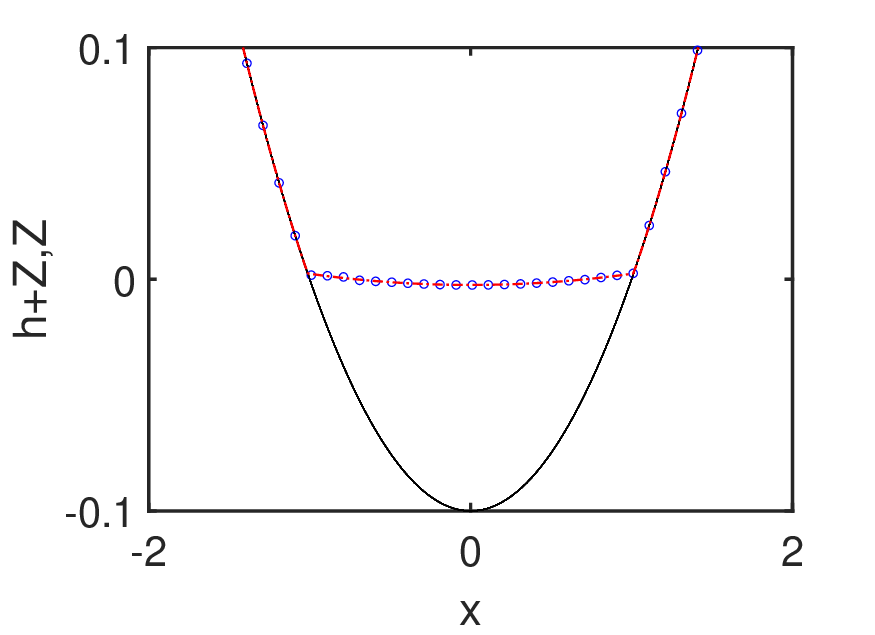}
	\includegraphics[width=0.35\columnwidth]{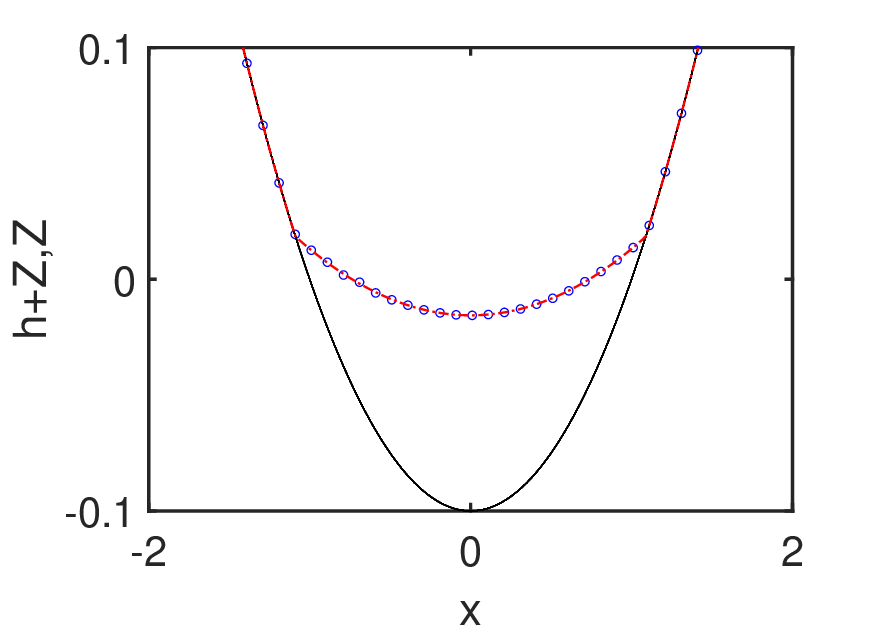}
	\includegraphics[width=0.35\columnwidth]{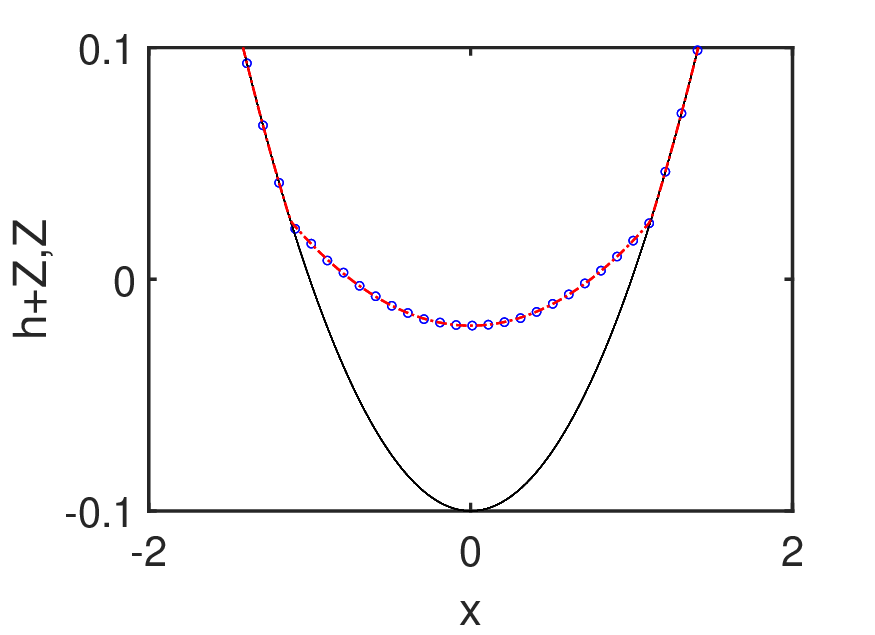}
	\includegraphics[width=0.35\columnwidth]{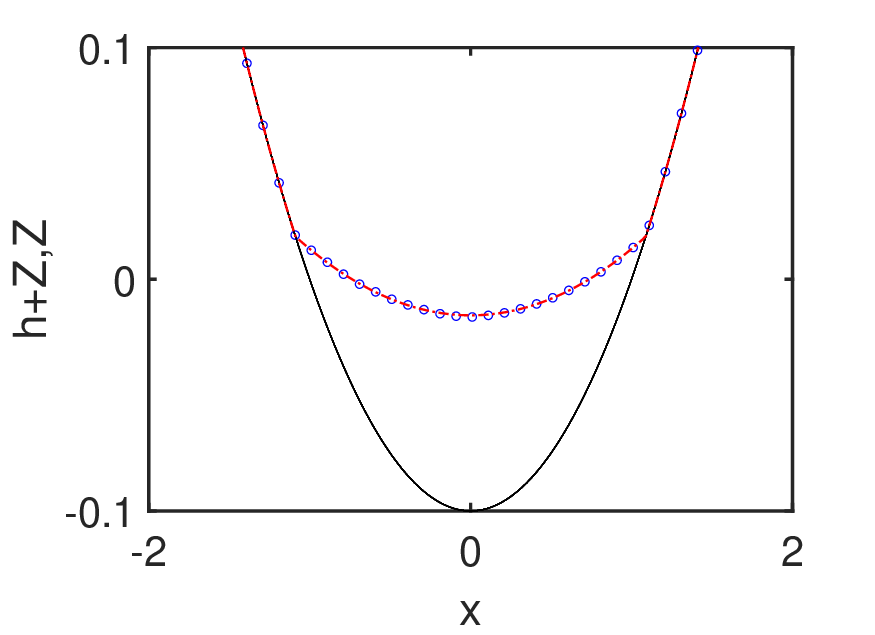}
	\includegraphics[width=0.35\columnwidth]{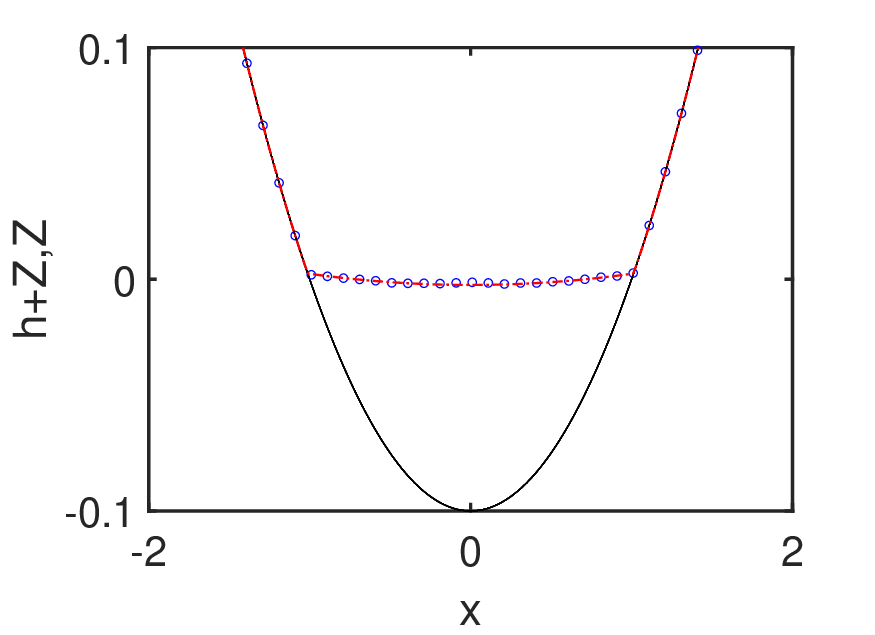}
	\includegraphics[width=0.35\columnwidth]{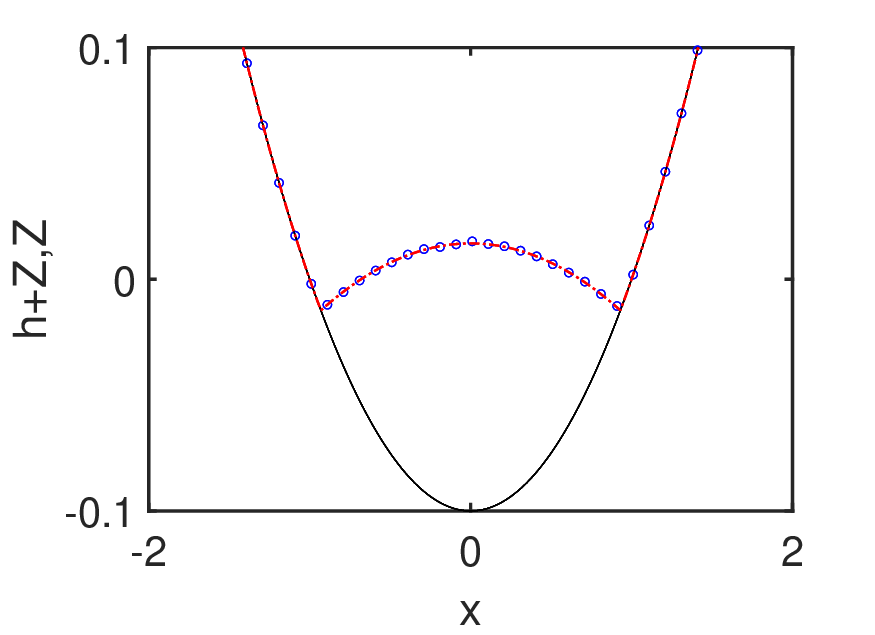}
	\includegraphics[width=0.35\columnwidth]{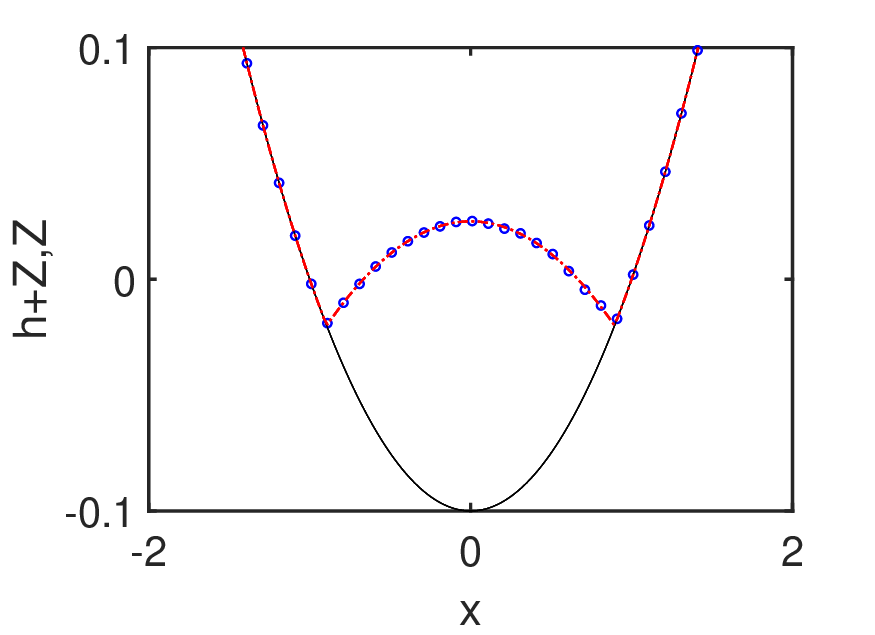}
	\caption{Vertical cross-section of the water surface level \( h + Z \) at \( y = 0 \) in a parabolic bowl at \( t =T/8, T/4, 3T/8, T/2, 5T/8, 3T/4, 7T/8 \) and \( T \) (from left to right and from top to bottom). Circles numerical solution; dashed line exact water surface; solid line domain geometry.}
	\label{FIG:2}
\end{figure}

\subsection{Parabolic Bowl}

This section adopts the shallow water equations to validate the positivity-preserving property of the proposed numerical scheme. We study the oscillating water surface in a two-dimensional parabolic bowl embedded within the square region \([-2,2]\times[-2,2]\). The parabolic geometry is defined by the following equation:

\begin{equation}\label{eq:43}Z(x,y)=h_0\left(\frac{x^2+y^2}{a^2}-1\right),\end{equation}
where \( h_0 \) and \( a \) are undetermined constants. The exact solution of the two-dimensional shallow water equations is given by
\begin{equation}\begin{aligned}\label{eq:44}
		& \eta(x,y,t)=\max\left\{Z,h_{0}\left[\frac{\sqrt{1-A^{2}}}{1-A\cos(\omega t)}-1-\frac{x^{2}+y^{2}}{a^{2}}\left(\frac{1-A^{2}}{(1-A\cos(\omega t))^{2}}-1\right)\right]\right\}, \\
		& u(x,y,t)=\frac{1}{1-A\cos(\omega t)}\left(\frac{1}{2}\omega x\sin(\omega t)\right),\quad\mathrm{if}\quad h>0, \\
		& v(x,y,t)=\frac{1}{1-A\cos(\omega t)}\left(\frac{1}{2}\omega y\sin(\omega t)\right),\quad\mathrm{if}\quad h>0,
\end{aligned}\end{equation}
where \(\eta=h+Z,\omega=\sqrt{8gh_{0}}/a,A=(a^{2}-r_{0}^{2})/(a^{2}+r_{0}^{2})\), and \( r_0 \) is a given constant indicating the initial radial distance from the center at time \( t = 0 \) \cite{Marche2007}.

In this numerical simulation, the computational domain \( [-2,2]\times[-2,2] \) is discretized into \( 200\times200 \) uniform cells. The exact solution (\ref{eq:44}) is applied as the initial condition at \( t = 0 \), with Dirichlet boundary conditions imposed along the domain edges. The constants are chosen as \( h_0 = 0.1 \), \( a = 1.0 \), and \( r_0 = 0.8 \). The method (\ref{eq:39}) is used to solve this problem. We plot the water surface level at \( t = T/8, T/4, 3T/8, T/2, 5T/8, 3T/4, 7T/8\) and \( T \) in Figure \ref{FIG:2}, where \(T=2\pi/\omega\) is the oscillation period. The numerical results show excellent agreement with the analytical solution, and the water depth remains strictly positive throughout the computation. In contrast, simulations without the positivity-preserving scheme would generate negative water depth values, leading to numerical instability or failure.

\begin{figure}
	\centering
	\includegraphics[width=.45\columnwidth]{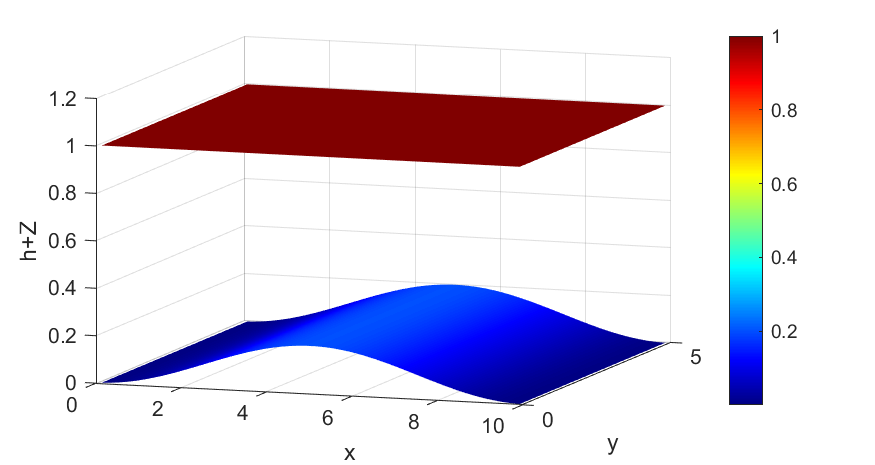}
	\includegraphics[width=.45\columnwidth]{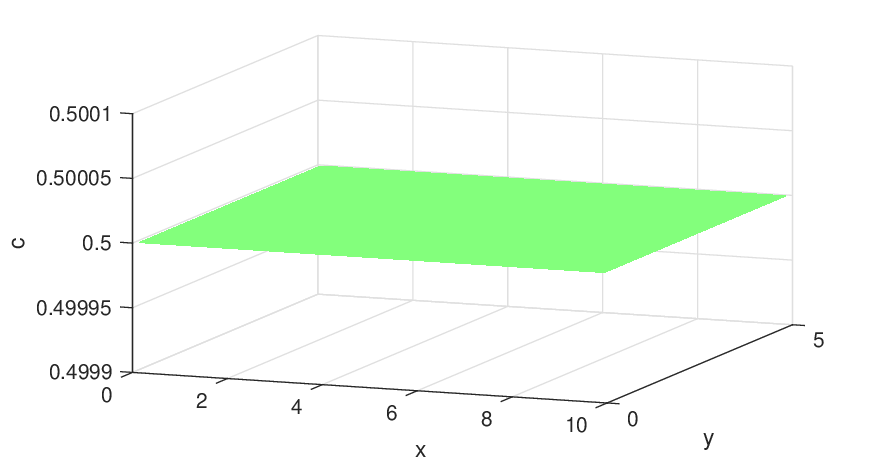}
	\caption{Three-dimensional view of the numerical solution at time \( t = 50 \). Left: The computed water surface elevation and the bottom topography. Right: The computed volumetric concentration.}
	\label{FIG:3}
\end{figure}

\subsection{Equilibrium preservation test for multi-component mixtures}

This section focuses on validating the well-balanced property of the proposed numerical scheme. In all test cases, we use the analytical solution (\ref{eq:15})-(\ref{eq:16}) derived in Section \ref{sec:wellb}. All test cases adopt a computational domain of \([0,10]\times[0,5]\) with the bottom topography defined as:

\begin{equation}\label{eq:45}Z(x,y)=0.1\left(1-\cos(2\pi x/10)\right).\end{equation}

In this section, we continue to employ a rectangular grid for numerical simulations, with  \( CFL=0.2 \).

\begin{table}
	\caption{\( L^1 \) error of the computed concentrations of the two constituents and the
		water surface elevation at time \( t = 50 \).}\label{tbl1}
\centering
	\begin{tabular}{c c c}
&&\\
\hline
		\(N_{x}\times N_{y}\) & \(E(c_{1})\) & \(E(c_{2})\) \\
\hline
		\(80\times40\) & \(1.3694\cdotp10^{-5}\) & \(9.7340\cdotp10^{-6}\) \\
		\(160\times40\) & \(3.8814\cdotp10^{-6}\) & \(2.7251\cdotp10^{-6}\)  \\
		\(320\times40\) & \(1.0431\cdotp10^{-6}\) & \(6.9773\cdotp10^{-7}\)  \\
		\(640\times40\) & \(2.8545\cdotp10^{-7}\) & \(1.5655\cdotp10^{-7}\)  \\
\hline
	\end{tabular}
	\label{TAL:1}
\end{table}

\begin{table}
	\caption{\( L^1 \) error of the computed concentrations of the four constituents and the
		water surface elevation at time \( t = 50 \).}\label{tbl1}
\centering
	\begin{tabular}{c c c c c}
&&&&\\
\hline
		\(N_{x}\times N_{y}\) & \(E(c_{1})\) & \(E(c_{2})\) & \(E(c_{3})\) & \(E(c_{4})\)\\
\hline
		\(80\times40\) & \(1.7203\cdotp10^{-5}\) & \(1.2240\cdotp10^{-5}\) & \(1.8877\cdotp10^{-6}\) & \(1.3962\cdotp10^{-6}\)\\
		\(160\times40\) & \(4.3205\cdotp10^{-6}\) & \(3.0382\cdotp10^{-6}\) & \(4.5464\cdotp10^{-7}\) & \(3.7260\cdotp10^{-7}\)\\
		\(320\times40\) & \(1.0983\cdotp10^{-6}\) & \(7.3661\cdotp10^{-7}\) & \(9.5674\cdotp10^{-8}\) & \(1.1619\cdotp10^{-7}\)\\
		\(640\times40\) & \(2.9260\cdotp10^{-7}\) & \(1.6115\cdotp10^{-7}\) & \(5.8892\cdotp10^{-9}\) & \(5.2061\cdotp10^{-8}\)\\
\hline
	\end{tabular}
	\label{TAL:2}
\end{table}

\begin{figure}
	\centering
	\includegraphics[width=.45\columnwidth]{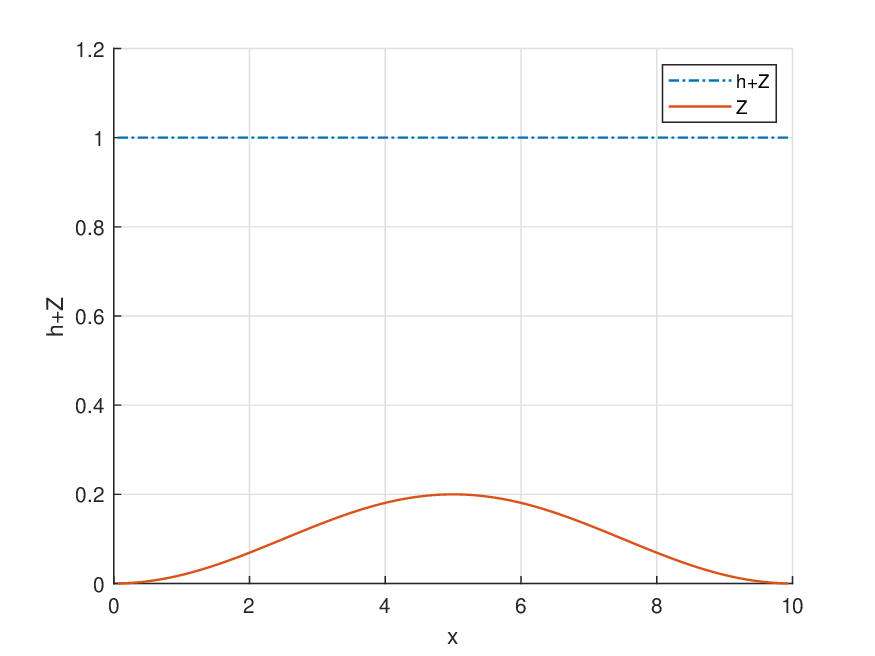}
	\includegraphics[width=.45\columnwidth]{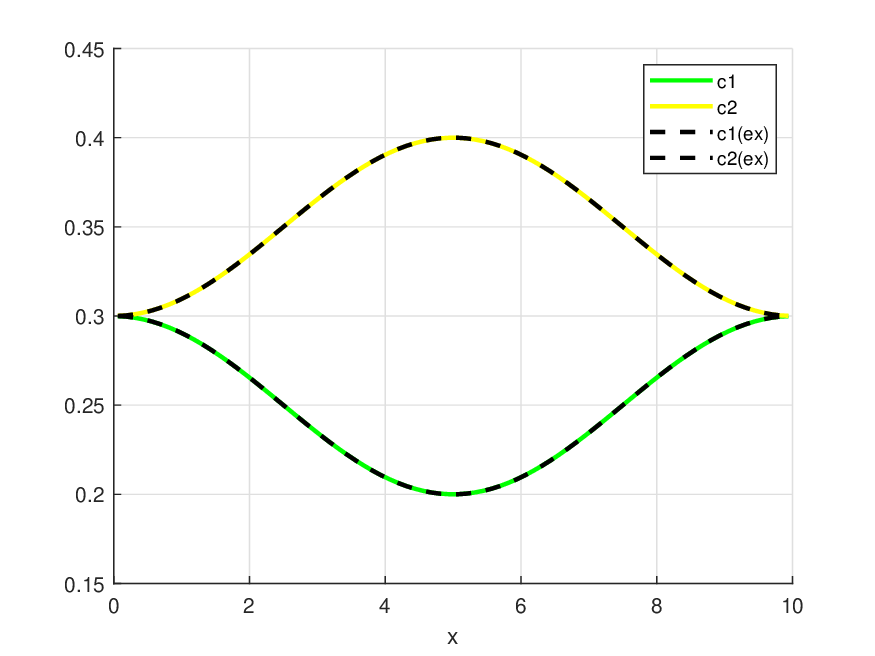}
	\caption{Cross sections at \( y = 2.5 \) of the computed volumetric concentrations for solutions with two constituents at time \( t = 50 \). Left: The computed water surface elevation and the bottom topography. Right: The computed volumetric concentration with two constituents.}
	\label{FIG:4}
\end{figure}

\begin{figure}
	\centering
	\includegraphics[width=.45\columnwidth]{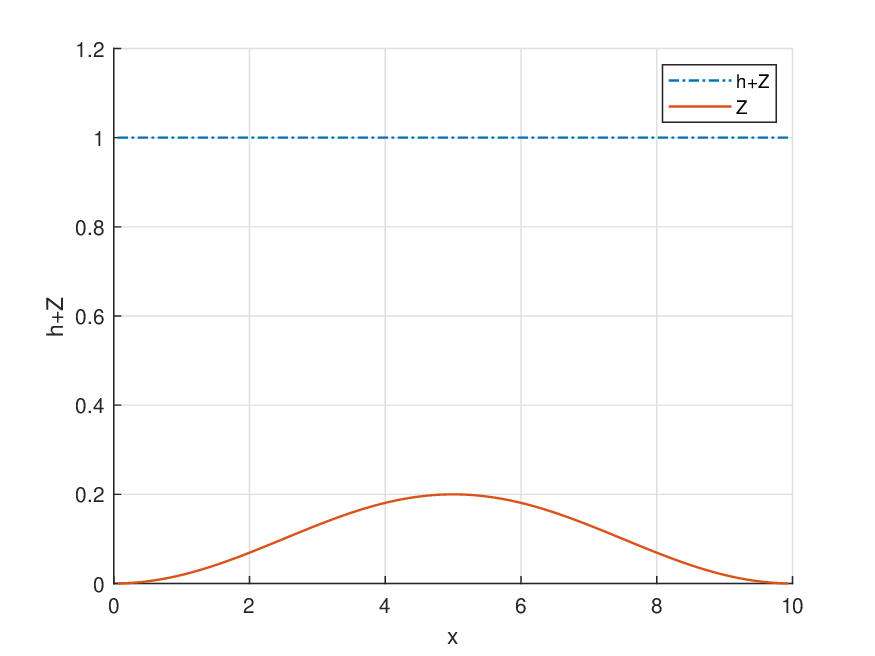}
	\includegraphics[width=.45\columnwidth]{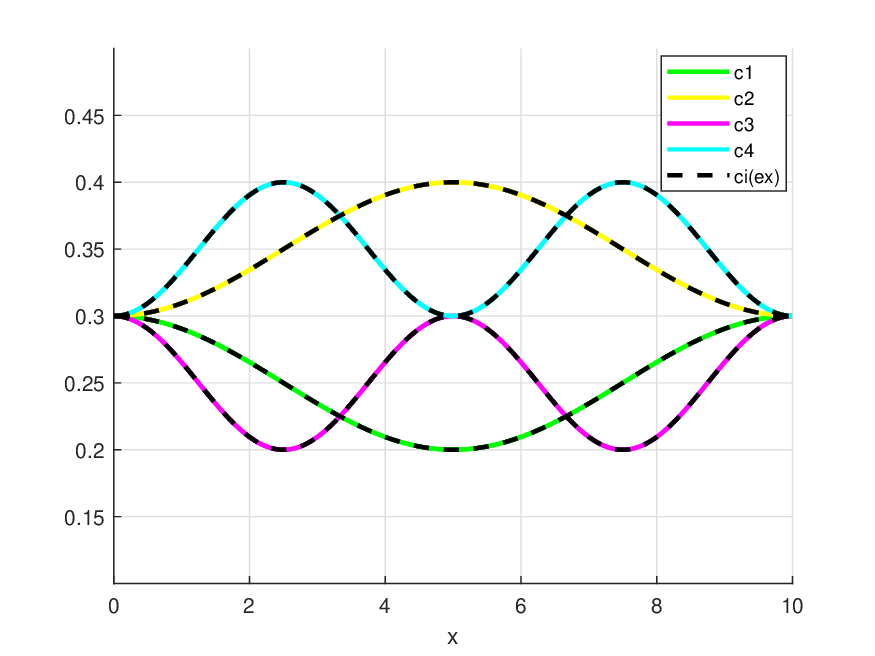}
	\caption{Cross sections at \( y = 2.5 \) of the computed volumetric concentrations for solutions with four constituents at time \( t = 50 \). Left: The computed water surface elevation and the bottom topography. Right: The computed volumetric concentration with four constituents.}
	\label{FIG:5}
\end{figure}

\subsubsection{Equilibrium of mono-constituent mixture}

This subsection aims to validate the equilibrium characteristics in a single-component mixing system. The initial conditions are set as follows: \( \eta = 1 \), \( u = v = 0 \), \( c_1 = 0.5 \), \( \Delta_1 = 0.2 \), resulting in an initial density of \( r = 1.1 \). Wall boundary conditions are applied on all computational domain boundaries. Figure \ref{FIG:3} shows the numerical solutions of the water surface elevation and mixture concentration at \( t = 50 \) obtained using the proposed scheme. As shown in the results, no numerical oscillations are observed in either the water surface elevation or the concentration field, with both quantities maintained at machine-level accuracy (approximately \( 10^{-15} \)). The numerical solutions remain stable throughout the computation, thereby validating the well-balanced property of the proposed numerical scheme.

\subsubsection{Equilibrium of multi-component mixture}

This subsection aims to validate the equilibrium characteristics in a multi-component mixing system. In this section, we presents equilibrium tests for systems containing two and four mixture components.

For the system containing two mixture components, the initial conditions are set as \( \eta = 1 \), \( u = v = 0 \), with the initial concentrations given by

\begin{equation}c_1(t,x,y)=0.2+0.1\cos^2(\pi x/10),\quad c_2(t,x,y)=0.3+0.1\sin^2(\pi x/10),\end{equation}
with relative densities \( \Delta_1 = \Delta_2 = 0.2 \).

Wall boundary conditions are applied to all domain boundaries. Figure \ref{FIG:4} displays the numerical solutions of the water surface elevation and concentration distributions at \( t = 50 \). As shown in the results, the water surface elevation exhibits no oscillations, and the numerical solutions for the two mixture concentrations show excellent agreement with the analytical solutions. To quantitatively assess the accuracy of the solutions, the \( L_1 \) error is defined as follows:

\begin{equation}E(u)=\frac{\left\|u^{ex}-u^{num}\right\|_{L_1}}{\left\|u^{ex}\right\|_{L_1}},\end{equation}
where \(\|u\|_{L_{1}}=\sum_{i,j}|C_{ij}||u_{ij}|\), \( u^{ex} \) is the exact solution and \( u^{num} \) is the approximate one. Table \ref{TAL:1} summarizes the \( L^1 \) error norms for both water surface elevation and scalar concentration fields in the benchmark tests. The table demonstrates that the error decreases with mesh refinement, confirming the convergence property of the proposed numerical scheme.

For the system containing four mixture components, the initial concentrations are specified as follows:
\begin{equation}\begin{gathered}
		c_{1}(t,x,y)=0.2+0.1\cos^{2}\left(\pi x/10\right),\quad c_{2}(t,x,y)=0.3+0.1\sin^{2}\left(\pi x/10\right),\\c_{3}(t,x,y)=0.2+0.1\cos^{2}\left(\pi x/20\right),\quad c_{4}(t,x,y)=0.3+0.1\sin^{2}\left(\pi x/20\right),
\end{gathered}\end{equation}
with relative densities \( \Delta_i = 0.2 \) for all mixture constituents \( i \) = 1, 2, 3, 4. In this test case, the water surface elevation is set to \( \eta = 1 \) and the vector field velocity is zero. This leads to a nontrivial steady state solution which satisfies  (\ref{eq:16}).

Wall boundary conditions are applied throughout the entire computational domain. Figure \ref{FIG:5} presents the water surface elevation and concentration distributions at \( t = 50 \). As shown in the results, the water surface elevation remains free of oscillations, and the numerical solutions for the four mixture concentrations show excellent agreement with the analytical solutions. Table \ref{TAL:2} provides the \( L_1 \) errors of the numerical solutions for the four mixture concentrations on different grids.

\begin{figure}
	\centering
	\includegraphics[width=.45\columnwidth]{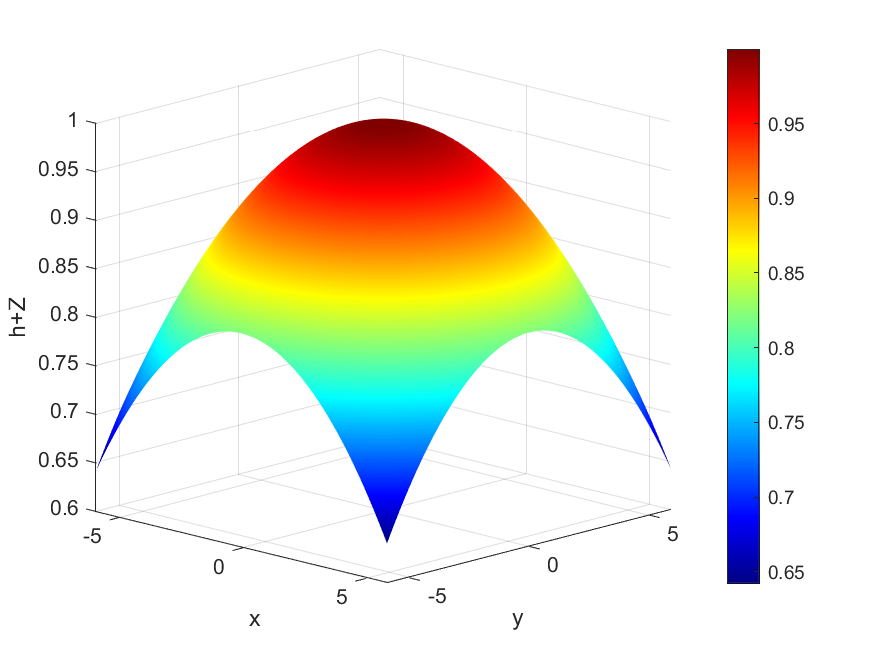}
	\includegraphics[width=.45\columnwidth]{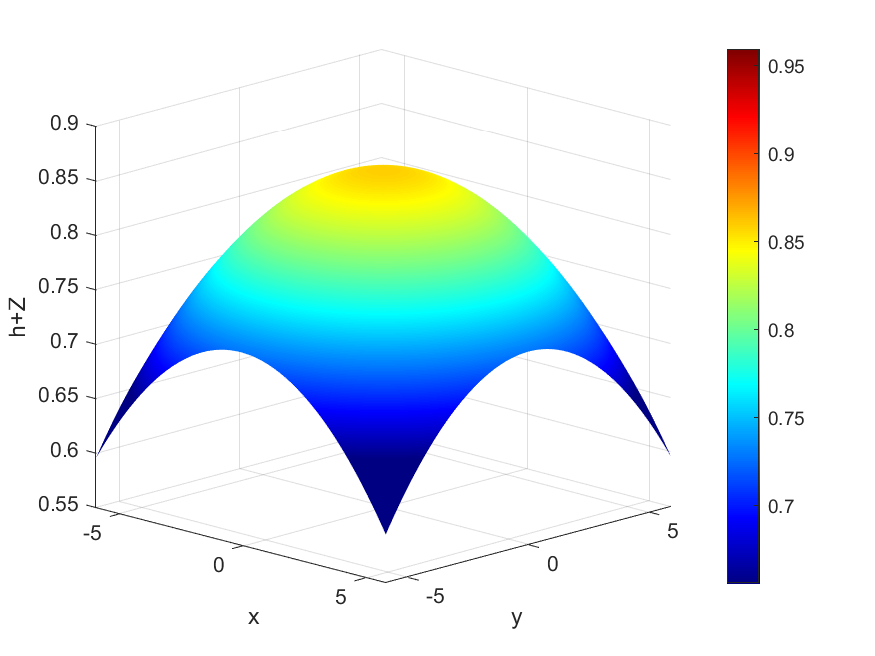}
	\caption{Three-dimensional views of the water surface elevation. Left: The initial condition. Right: The computed solution at time \( t = 4 \).}
	\label{FIG:6}
\end{figure}

\begin{figure}
	\centering
	\includegraphics[width=.45\columnwidth]{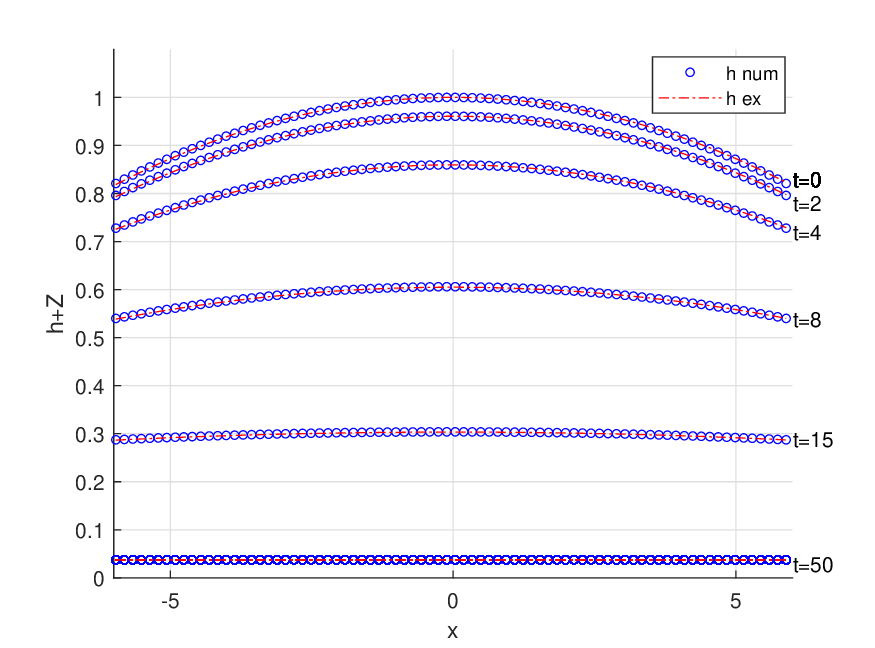}
	\includegraphics[width=.45\columnwidth]{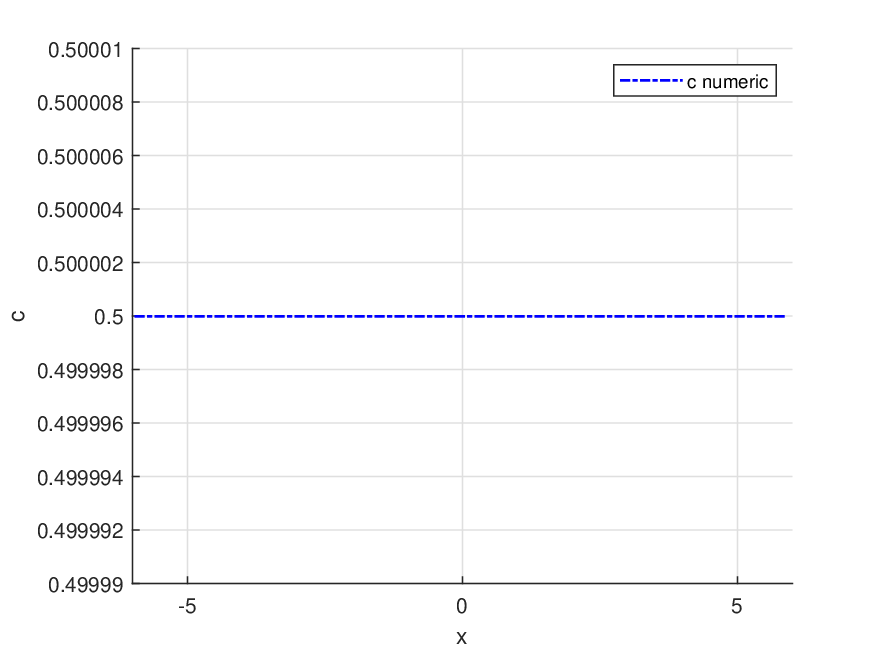}
	\caption{Cross section at \( y = 0 \) of the numerical solutions. Left: The computed and analytical water depth at different times. Circles numerical solution; dashed line exact water surface. Right: The computed scalar concentration at time \( t = 50 \).}
	\label{FIG:7}
\end{figure}

\begin{figure}
	\centering
	\includegraphics[width=.65\columnwidth]{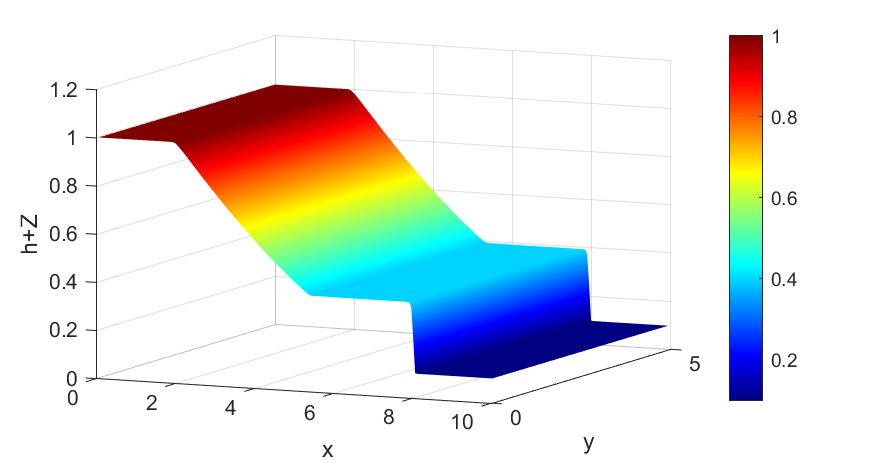}
	\includegraphics[width=.65\columnwidth]{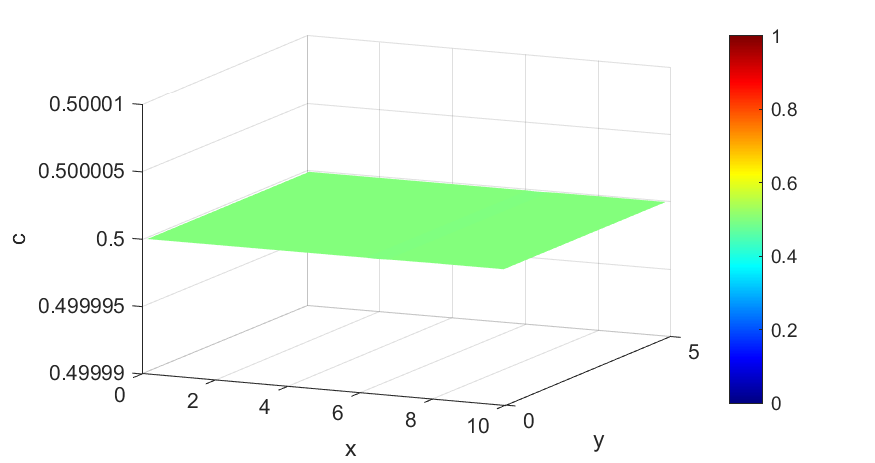}
	\caption{Dam-break with scalar transport for constant concentration. The computed solutions at time \( t = 3 \).}
	\label{FIG:8}
\end{figure}

\subsection{Parabolic flood waves with mono-constituent mixture}

In this section, we introduces a modified version of the Thacker analytical solution \cite{Carlisle1981}, which used to simulate the motion of parabolic flood waves on a frictionless planar seabed topography. The purpose is to verify the preservation of the constant concentration state in the computational process and the consistency of the discretization of the transport equation and the continuity equation. In the numerical examples of this section, we adopt the following analytical solution:
\begin{equation}\begin{aligned}
		&q_{1}(t,x,y)=h_{0}r\left[\frac{T^{2}}{t^{2}+T^{2}}-\frac{x^{2}+y^{2}}{R_{0}^{2}}\left(\frac{T^{2}}{t^{2}+T^{2}}\right)^{2}\right],\\
		&	q_{2}(t,x,y)=\frac{xt}{t^{2}+T^{2}}q_{1}(t,x,y),\\
		&	q_{3}(t,x,y)=\frac{yt}{t^{2}+T^{2}}q_{1}(t,x,y),
\end{aligned}\end{equation}
where \( c(t, x, y) = c_0 \) is a constant, \( h_0 \) is the initial peak height of the paraboloidal water surface, \( R_0 \) is the initial radius of the paraboloidal mound [45], and \( T \) is the reference time given by \( T = R_0 / \sqrt{2gh_0} \).

The proposed numerical scheme is used to solve system (\ref{eq:3}). The initial conditions: \( h = h_0 \left(1 - (x^2 + y^2)/R_0^2\right) \), \( u = v = 0 \), \( c = 0.5 \), and \( \Delta = 0.2 \), yielding \( r = 1.1 \). The parameters are configured as \( h_0 = 1 \) and \( R_0 = 14 \). The computational domain \( [-6, 6] \times [-6, 6] \) is discretized into \( 80 \times 80 \) uniform cells with a CFL condition of 0.2. The analytical solution serves as the boundary condition applied in the numerical test.

Figure \ref{FIG:6} illustrates the initial condition of the flood wave and the three-dimensional view of the water depth at \( t = 4 \). The numerical solution for the water depth maintains symmetry without oscillations, and the \( L_1 \) error is on the order of \( 1.6\cdot 10^{-3} \).

Figure \ref{FIG:7} presents the evolution of the computed and exact solutions for water depth at different time instances, along with the numerical results for mixture concentrations. At all simulated time steps, the numerical solutions for water depth show excellent agreement with the analytical solutions while maintaining constant concentrations. This confirms the satisfactory consistency in the discretization treatment between the transport equation and the continuity equation.

\subsection{Dam-break with scalar transport}

This section simulates the rapidly changing flow field during a dam-break process, aiming to investigate the consistency between the transport equation and the continuity equation at the discrete level under rapidly varying flow conditions. The computational domain \([0,10]\times[0,5]\) is divided into \(200\times50\) uniform elements.

We consider the following initial condition:
\begin{equation}h(0,x,y)=
	\begin{cases}
		1 & x\leq5 ,\\
		0.1 & \mathrm{otherwise}, 
\end{cases}\end{equation}
and \( u(0,x,y)=v(0,x,y)=0 \), \( c(0,x,y)=0.5 \), the outflow boundary conditions are imposed in the \( x \)- and \( y \)- directions except for the variable \( p_3 \), in the \( y \)- direction, for which the Dirichlet boundary condition \( p_3 = 0 \) is imposed \( (i.e, v = 0) \).

As illustrated in Figure {\ref{FIG:8}}, the computed water surface elevation agrees well with the solution in \cite{Hanini2021}, and the concentration remains constant. This verifies the consistency between the continuity equation and the transport equation at the discrete level under rapidly changing flow conditions.

\section{Conclusions}\label{sec:concl}

In this paper, we propose a robust numerical scheme for simulating coupled systems of variable-topography shallow water flow and scalar transport. Through a series of numerical experiments, we validate the well-balanced property of the proposed scheme, confirming its capabilities in preserving the positivity of both water depth and scalar concentration, while demonstrating the consistency in discretization between the continuity equation and the transport equation. This method is designed based on structured grids, consequently, it currently addresses shallow water wave problems only in simple domains for the two-dimensional case. Future work will focus on extending the method to complex geometries by developing numerical algorithms utilizing unstructured grids.

\section*{Acknowledgments}
The research of M.Li was supported in part by the NSFC 12271082 62231016.



\begin{thebibliography}{99}

\bibitem{Luis2012} C. Luis, M. E. Vázquez-Cendón. Unstructured finite volume discretisation of bed friction and convective flux in solute transport models linked to the
shallow water equations. J Comput Phys, 231(8), 3317-3339 (2012).

\bibitem{Faranak2018} B. Faranak, S. Behrouz, C. Newman James. Solution of fully-coupled shallow water equations and contaminant transport using a primitive-variable Riemann method. Environ Fluid Mech, 18(2), 515-535 (2018).

\bibitem{Hu2020} D. Hu, S. Yao, C. Duan, S. Li. Real-time simulation of hydrodynamic and scalar transport in large river-lake systems. J Hydrol, 582, 124531 (2020).

\bibitem{Pablo2018} O. Pablo, F. Bruno, V. Nicolo, A. Athanasious, S. Thorsten, G. Carlo. Instantaneous transport of a passive scalar in a turbulent separated flow. Environ Fluid Mech , 18(2), 487-513 (2018).

\bibitem{Vázquez-Cendón1999} M. E. Vázquez-Cendón. Improved treatment of source terms in upwind
schemes for the shallow water equations in channels with irregular geometry. J Comput Phys,  148(2), 497-526 (1999).

\bibitem{Thierry2003} G. Thierry, H. Jean-Marc, S. Nicolas. Some approximate Godunov
schemes to compute shallow-water equations with topography. Comput $\&$ Fluids, 32(4), 479-513 (2003).

\bibitem{Abdelaziz2016}B. Abdelaziz, M. Abdolmajid, K. Alexander. Well-balanced positivity preserving cell-vertex central-upwind scheme for shallow water flows. Comput $\&$ Fluids, 136, 193-206 (2016).

\bibitem{Khorshid2017}S. Khorshid, A. Mohammadian, I. Nistor. Extension of a well-balanced central upwind scheme for variable density shallow water flow equations on triangular grids. Comput $\&$ Fluids, 156, 441-448 (2017).

\bibitem{Guerrero Fernández2020}E. Guerrero Fernández, M. J. Castro-Díaz, T. Morales de Luna. A Second-Order Well-Balanced Finite Volume Scheme for the Multilayer Shallow Water Model with Variable Density. Mathematics, 8(5), 848 (2020).

\bibitem{Hanini2021}A. Hanini, A. Beljadid, D. Ouazar. A well-balanced positivity-preserving numerical scheme for shallow water models with variable density. Comput $\&$ Fluids, 231, 105156 (2021).

\bibitem{Ghazizadeh2022}M.A Ghazizadeh, A. Mohammadian. An adaptive central-upwind scheme on quadtree grids for variable density shallow water equations. Int J Numer Meth Fluids, 94, 461-481 (2022).

\bibitem{Reed1973}W. H. Reed and T.R. Hill. Triangular mesh methods for the neutron transport equation. Tech. Report LA-UR-73-479, Los Alamos Scientifific Laboratory (1973).

\bibitem{Cockburn1990}B. Cockburn, S. Hou and C.-W. Shu. The Runge-Kutta local projection discontinuous Galerkin fifinite element method for conservation laws IV: the multidimensional case. Mathematics of Computation, 54, 545-581 (1990).

\bibitem{Cockburn1989-1}B. Cockburn, S. Y. Lin and C.-W. Shu. TVB Runge-Kutta local projection discontinuous Galerkin fifinite element method for conservation laws III: one dimensional systems. Journal of Computational Physics, 84, 90-113 (1989).

\bibitem{Cockburn1989-2}B. Cockburn and C.-W. Shu. TVB Runge-Kutta local projection discontinuous Galerkin fifinite element method for conservation laws II: general framework. Mathematics of Computation, 52, 411-435 (1989).

\bibitem{Cockburn1991}B. Cockburn and C.-W. Shu. The Runge-Kutta local projection P1-discontinuousGalerkin fifinite element method for scalar conservation laws. Mathematical Modelling and Numerical Analysis (M2AN), 25, 337-361 (1991).

\bibitem{Cockburn1998-1}B. Cockburn and C.-W. Shu. The Runge-Kutta discontinuous Galerkin method for conservation laws V: multidimensional systems. Journal of Computational Physics, 141, 199-224 (1998).

\bibitem{Li2014}M. Li, P. Guyenne, F. Li, L. Xu. High order well-balanced CDG-FE methods for shallow water waves by a Green-Naghdi mode. Journal of Computational Physics, 257, 169-192 (2014).

\bibitem{Li2017}M. Li, P. Guyenne, F. Li, L. Xu. A Positivity-Preserving Well-Balanced Central Discontinuous Galerkin Method for the Nonlinear Shallow Water Equations. J Sci Comput, 71, 994-1034 (2017).

\bibitem{Guerrero Fernández2022}E. Guerrero Fernández, M. J. Castro-Díaz, M. Dumbser, T. Morales de Luna. An Arbitrary High Order Well-Balanced ADER-DG Numerical Scheme for the Multilayer Shallow-Water Model with Variable Density. Journal of Scientific Computing, 90, 52 (2022).

\bibitem{Li2022}M. Li, R. Mu, H. Dong. A well-balanced discontinuous Galerkin method for the shallow water flows on erodible bottom. Computers and Mathematics with Applications, 119, 13-20 (2022).

\bibitem{Zhang2023}J. Zhang, Y. Xia, Y. Xu. Moving Water Equilibria Preserving Discontinuous Galerkin Method for the Shallow Water Equations. Journal of Scientific Computing, 95, 48 (2023).

\bibitem{Ersing2024}P. Ersing, A. R. Winters. An Entropy Stable Discontinuous Galerkin Method for the Two-Layer Shallow Water Equations on Curvilinear Meshes. Journal of Scientific Computing, 98, 62 (2024).

\bibitem{Bunya2009}S. Bunya, E. J. Kubatko, J. J. Westerink, C. Dawson. A wetting and drying treatment for the Runge-Kutta discontinuous Galerkin solution to the shallow water equations. Comput. Meth. Appl. Mech. Eng, 198, 1548-1562 (2009).

\bibitem{Xing2010}Y. Xing, X. Zhang, C.-W. Shu. Positivity-preserving high order well-balanced discontinuous Galerkin methods for the shallow water equations. Advances in Water Resources, 33, 1476-1493 (2010).

\bibitem{Xian2021}W. Xian, A. Chen, Y. Cheng, H. Dong, Numerical simulations for shallow water flows	over erodible bed by central DG methods, Int. J. Numer. Anal. Model, 18(2), 143-164 (2021).

\bibitem{Javier2012}M. Javier, L. Borja, García-Navarro Pilar. A Riemann solver for
unsteady computation of 2D shallow flows with variable density. J Comput Phys, 231(14), 4775-4807 (2012).

\bibitem{Cockburn1998-2}B. Cockburn, C.-W. Shu, The Runge-Kutta discontinuous Galerkin method for conservation laws V multidimensional systems, J. Comput. Phys, 141, 199-224 (1998).

\bibitem{Marche2007}F. Marche, P. Bonneton, P. Fabrie, N. Seguin. Evaluation of well-balanced bore-capturing schemes for 2D wetting and drying processes. Int. J. Numer. Methods Fluids 53, 867-894 (2007).

\bibitem{Carlisle1981}T. W. Carlisle. Some exact solutions to the nonlinear shallow-water wave equations. J Fluid Mech, 107, 499-508 (1981).

\end{thebibliography}
\end{document}